\documentclass{article}

\usepackage{amssymb,amsmath,color,enumerate,ascmac,latexsym,diagbox}
\usepackage[all]{xy}
\usepackage[abbrev]{amsrefs}
\usepackage[mathscr]{eucal}
\usepackage[top=30truemm,bottom=30truemm,left=25truemm,right=25truemm]{geometry}
\usepackage{mathrsfs}
\usepackage{amsthm}
\usepackage[OT2,T1]{fontenc}
\usepackage[russian,english]{babel}

\newcommand{\kg}{{\vspace{0.2in}}}

\newcommand{\Q}{\mathbb{Q}}

\newcommand{\Z}{\mathbb{Z}}
\newcommand{\R}{\mathbb{R}}
\newcommand{\Co}{\mathbb{C}}

\newcommand{\five}{~~~~~}

\newcommand{\mm}[1]{\mathop{\mathrm{#1}}}

\newcommand{\brkt}[1]{\langle #1 \rangle}

\newcommand{\bber}[2]{B_{#1}^{(#2)}}
\newcommand{\cber}[2]{C_{#1}^{#2}}

\newcommand{\pdv}[2]{\frac{{\partial}^{#1}}{\partial {#2}^{#1}}}

\newcommand{\ctext}[1]{\raise0.2ex\hbox{\textcircled{\scriptsize{#1}}}}
\newcommand{\cy}{\textup{\foreignlanguage{russian}{\cyrsh}}}


\usepackage{graphicx}
\usepackage{stmaryrd}
\makeatletter
\def\mapstofill@{%
	\arrowfill@{\mapstochar\relbar}\relbar\rightarrow}
\newcommand*\xmapsto[2][]{%
	\ext@arrow 0395\mapstofill@{#1}{#2}}

\usepackage[dvipdfmx]{hyperref}

\usepackage{xcolor}
\hypersetup{
	colorlinks=true,
	citebordercolor=green,
	linkbordercolor=red,
	urlbordercolor=cyan,
}

\theoremstyle{definition}
\newtheorem{theorem}{Theorem}[section]
\newtheorem{definition}[theorem]{Definition}
\newtheorem{lemma}[theorem]{Lemma}
\newtheorem{proposition}[theorem]{Proposition}
\newtheorem{corollary}[theorem]{Corollary}
\newtheorem{remark}[theorem]{Remark}
\newtheorem{example}[theorem]{Example}

\makeatletter
\def\underbrace@#1#2{\vtop {\m@th \ialign {##\crcr $\hfil #1{#2}\hfil $\crcr \noalign {\kern 3\p@ \nointerlineskip }\upbracefill \crcr \noalign {\kern 3\p@ }}}}
\def\overbrace@#1#2{\vbox {\m@th \ialign {##\crcr \noalign {\kern 3\p@ }\downbracefill \crcr \noalign {\kern 3\p@ \nointerlineskip }$\hfil #1 {#2}\hfil $\crcr }}}

\def\underbrace#1{%
	\mathop{\mathchoice{\underbrace@{\displaystyle}{#1}}
		{\underbrace@{\textstyle}{#1}}
		{\underbrace@{\scriptstyle}{#1}}
		{\underbrace@{\scriptscriptstyle}{#1}}}\limits
}
\def\overbrace#1{%
	\mathop{\mathchoice{\overbrace@{\displaystyle}{#1}}
		{\overbrace@{\textstyle}{#1}}
		{\overbrace@{\scriptstyle}{#1}}
		{\overbrace@{\scriptscriptstyle}{#1}}}\limits
}

\makeatother

\allowdisplaybreaks

\begin{document}
	
	\title{On generalization of duality formulas for the Arakawa-Kaneko type zeta functions}
	\author{Kyosuke Nishibiro}
	\date{}
	\maketitle
	
	\noindent{\bf Abstract:} 
	Kaneko and Tsumura introduced the Arakawa-Kaneko type zeta function $\eta(-k_1,\ldots,-k_r;s_1,\ldots,s_r)$ for non-negative integers $k_1,\ldots,k_r$ and complex variables $s_1,\ldots,s_r$. Recently, Yamamoto showed that, by using the multiple integral expression, $\eta(u_1,\ldots,u_r;s_1,\ldots,s_r)$ can be extended to an analytic function of 2$r$ variables. Also, he showed that the function $\eta(u_1,\ldots,u_r;s_1,\ldots,s_r)$ satisfies a duality formula. In this paper, by using the a generalization of non-strict multi-indexed polylogarithm, we define a kind of the Arakawa-Kaneko type zeta function, and show that this function satisfies a certain duality formula.
	
	\noindent{\bf Keywords:} Multivariable Arakawa-Kaneko zeta function, poly-Bernoulli numbers.

	\section{Introduction}
	
	\footnote[0]{{\bf 2010 Mathematical Subject Classification:} primary:11M32; secondary:11B68.}
	
	In \cite{AK1}, Arakawa and Kaneko introduced the function
	\[
	\xi(\mathbf{k}_r;s)=\xi(k_1,\ldots,k_r;s)=\frac{1}{\Gamma(s)}\int_{0}^{\infty} t^{s-1}\frac{\mm{Li}_{\mathbf{k}_r}(1-e^{-t})}{e^t-1} dt
	\]
	for $\mathbf{k}_r=(k_1,\ldots,k_r)\in\Z_{\geq1}^r$ and $s\in\Co$ with $\mm{Re}(s)>0$, which is called the Arakawa-Kaneko zeta function. Here, 
	\[
	{\mm{Li}}_{\mathbf{k}_r}(z)={\mm{Li}}_{k_1,\ldots,k_r}(z)=\sum_{0<m_1<\cdots<m_r}^{}\frac{z^{m_r}}{m_1^{k_1}\cdots m_r^{k_r}}~(|z|<1)
	\]
	is the multiple polylogarithm. They also showed that the function $\xi(\mathbf{k}_r;s)$ can be continued analytically to the whole plane $\Co$. As a relative of $\xi(\mathbf{k}_r;s)$, in \cite{KT1}, Kaneko and Tsumura introduced the function
	\[
	\eta(\mathbf{k}_r;s)=\eta(k_1,\ldots,k_r;s)=\frac{1}{\Gamma(s)}\int_{0}^{\infty} t^{s-1}\frac{\mm{Li}_{\mathbf{k}_r}(1-e^{t})}{1-e^t} dt
	\]
	for $s\in\Co$ with $\mm{Re}(s)>0$, $\mathbf{k}_r\in\Z_{\geq1}^r$ or $\mathbf{k}_r\in\Z_{\leq0}^r$, and showed that the function $\eta(\mathbf{k}_r;s)$ can be continued analytically to the whole plane $\Co$. The function $\eta(\mathbf{k}_r;s)$ is considered to be a twin sibling of the function $\xi(\mathbf{k}_r;s)$. It is shown that the values of these functions at non-positive integers can be expressed in terms of the multi-poly-Bernoulli numbers 
	\begin{align*}
		\frac{\mm{Li}_{\mathbf{k}_r}(1-e^{-t})}{1-e^{-t}}&=\sum_{m=0}^{\infty}B_{m}^{(\mathbf{k}_r)}\frac{t^m}{m!},\\
		\frac{\mm{Li}_{\mathbf{k}_r}(1-e^{-t})}{e^t-1}&=\sum_{m=0}^{\infty}C_{m}^{(\mathbf{k}_r)}\frac{t^m}{m!},
	\end{align*}
	and that the values of these functions at positive integers are closely related to the multiple zeta values
	\begin{align*}
		\zeta(\mathbf{k}_r)=\zeta(k_1,\ldots,k_r)=\sum_{0<m_1<\cdots<m_r}^{}\frac{1}{m_1^{k_1}\cdots m_r^{k_r}}
	\end{align*}
	for $\mathbf{k}_r\in\Z_{\geq1}^r$ with $k_r\geq2$ (for details, see \cite{AK1, KT1}). 
	
	Analogously, for $\mathbf{u}_r=(u_1,\ldots,u_r)\in\Co^r$ and $d\in\{1,\ldots,r\}$, Kaneko and Tsumura introduced the multi-indexed poly-Bernoulli numbers by the generating function
	\begin{align}\label{mulbber}
		\frac{{\mm{Li}}_{\mathbf{u}_r}^{\textup{\foreignlanguage{russian}{\cyrsh}}}(1-e^{-t_1-\cdots-t_r},\ldots,1-e^{-t_r})}{(1-e^{-t_1-\cdots-t_r})\cdots(1-e^{-t_d-\cdots-t_r})}=\sum_{m_1,\ldots,m_r=0}^{\infty} B_{\mathbf{m}_r}^{(\mathbf{u}_r),(d)}\frac{t_1^{m_1}}{m_1!}\cdots\frac{t_r^{m_r}}{m_r!},
	\end{align}
	and defined the related Arakawa-Kaneko type zeta function for $-\mathbf{k}_r=(-k_1,\ldots,-k_r)\in\Z_{\geq0}^r$ by
	\begin{align}\label{kteta}
		\eta(-\mathbf{k}_r;\mathbf{s}_r)=\prod_{j=1}^{r}\frac{1}{\Gamma(s_j)}\int_{0}^{\infty} t_1^{s_1-1}\cdots t_r^{s_r-1}\frac{\mm{Li}_{-\mathbf{k}_r}^{\cy}(1-e^{t_1+\cdots+t_r},\ldots,1-e^{t_r})}{(1-e^{t_1+\cdots+t_r})\cdots(1-e^{t_r})} dt_1\cdots dt_r.
	\end{align}
	Here, for $\mathbf{u}_r, \mathbf{z}_r=(z_1,\ldots,z_r)\in\Co^r$ with $|z_j|<1~(j=1,\ldots,r)$,
	\[
	{\mm{Li}}_{\mathbf{u}_r}^{\cy}(\mathbf{z}_r)={\mm{Li}}_{u_1,\ldots,u_r}^{\cy}(z_1,\ldots,z_r)=\sum_{l_1,\ldots,l_r=1}^{\infty}\frac{z_1^{l_1}\cdots z_r^{l_r}}{l_1^{u_1}\cdots(l_1+\cdots+l_r)^{u_r}}
	\]
	is the multiple polylogarithm of $\cy$-type. In \cite{KT1}, Kaneko and Tsumura showed that the duality formula
	\[
	\bber{\mathbf{n}_r}{-\mathbf{k}_r),(r}=\bber{\mathbf{k}_r}{-\mathbf{n}_r),(r}~~(\mathbf{k}_r, \mathbf{n}_r\in\Z_{\geq0}^r)
	\]
	holds, which is a generalization of that for poly-Bernoulli numbers $\bber{n}{-k}$. Recently, in \cite{Ya}, Yamamoto showed that the function $\eta(-\mathbf{k}_r;\mathbf{s}_r)$ can be extended to an analytic function of $2r$ variables $\eta(\mathbf{u}_r;\mathbf{s}_r)$. Also, in \cite{Ya}, he conjectured that
	\[
	\eta(\mathbf{k}_r;\mathbf{n}_r)\in\mathcal{Z}
	\]
	holds for $\mathbf{k}_r, \mathbf{n}_r\in\Z_{\geq1}$. Here,
	\[
	\mathcal{Z}=\brkt{~\zeta(\mathbf{k}_r)~|~\mathbf{k}_r\in\Z_{\geq1}^r, k_r\geq2~}_{\Q}
	\] 
	is the $\Q$-linear space spanned by all multiple zeta values for admissible indices $\mathbf{k}_r$. This conjecture was solved by Brown \cite{Bro}, and later by Ito and Sato \cite{It3}.
	In particular, Ito and Sato showed that
	\[
	\eta(\mathbf{k}_r;\mathbf{n}_r)\in\mathcal{Z}_{k_1+\cdots+k_r+n_1+\cdots+n_r}
	\]
	holds. Here, for $k\in\Z_{\geq0}$, 
	\[
	\mathcal{Z}_{k}=\brkt{~\zeta(\mathbf{k}_r)~|~\mathbf{k}_r\in\Z_{\geq1}^r, k_r\geq2, k_1+\cdots+k_r=k~}_{\Q}
	\]
	is the $\Q$-linear subspace of $\mathcal{Z}$ spanned by all multiple zeta values for admissible indices of weight $k$. As an analogue of ${\mm{Li}}_{\mathbf{u}_r}^{\cy}(\mathbf{z}_r)$, Baba, Nakasuji and Sakata introduced the multiple polylogarithm ${\mm{Li}}_{\mathbf{u}_r}^{\cy,\star}(\mathbf{z}_r)$ and multi-indexed poly-Bernoulli numbers $\mathbb{B}_{\mathbf{n}_r}^{\star,(\mathbf{u}_r)}$ as follows.
	
	\begin{definition}[{\cite{BS}}]
		For $\mathbf{u}_r, \mathbf{z}_r\in\Co^r$ with $|z_j|<1~(j=1,\ldots,r)$, the multiple polylogarithm ${\mm{Li}}_{\mathbf{u}_r}^{\cy,\star}(\mathbf{z}_r)$ and multi-indexed poly-Bernoulli numbers $\mathbb{B}_{\mathbf{n}_r}^{\star,(\mathbf{u}_r)}$ are defined by 
		\begin{align*}
			{\mm{Li}}_{\mathbf{u}_r}^{\cy,\star}(\mathbf{z}_r)&=\sum_{l_1=1,l_2,\ldots,l_r=0}^{\infty}\frac{z_1^{l_1}\cdots z_r^{l_r}}{l_1^{u_1}\cdots(l_1+\cdots+l_r)^{u_r}},\\
			\frac{{\mm{Li}}_{\mathbf{u}_r}^{\textup{\foreignlanguage{russian}{\cyrsh}},\star}(1-e^{-t_1-\cdots-t_r},\ldots,1-e^{-t_r})}{(1-e^{-t_1-\cdots-t_r})\cdots(1-e^{-t_r})}&=\sum_{m_1,\ldots,m_r=0}^{\infty} \mathbb{B}_{\mathbf{m}_r}^{\star,(\mathbf{u}_r)}\frac{t_1^{m_1}}{m_1!}\cdots\frac{t_r^{m_r}}{m_r!},
		\end{align*}
		respectively.
	\end{definition}
	Note that ${\mm{Li}}_{\mathbf{u}_r}^{\cy,\star}(\mathbf{z}_r)$ is a $\cy$-type analogue of the non-strict multiple polylogarithm
	\[
	{\mm{Li}}_{\mathbf{k}_r}^{\star}(z)=\sum_{0<m_1\leq\cdots\leq m_r}^{}\frac{z^{m_r}}{m_1^{k_1}\cdots m_r^{k_r}}~(|z|<1),
	\]
	since
	\[
	{\mm{Li}}_{\mathbf{k}_r}^{\star}(z)={\mm{Li}}_{\mathbf{k}_r}^{\cy,\star}(z,\ldots,z)
	\]
	holds. In \cite{BS}, Baba, Nakasuji and Sakata showed various relations among $\mathbb{B}_{\mathbf{n}_r}^{\star,(\mathbf{u}_r)}$. Inspired by their results, we consider the following polylogarithm, Bernoulli numbers and Arakawa-Kaneko type $\eta$ function.
	\begin{definition}
		For $\mathbf{u}_r, \mathbf{z}_r\in\Co^r$ with $|z_j|<1~(j=1,\ldots,r)$ and $\mathbf{a}_r=(a_1,\ldots,a_r)\in\Z_{\geq0}^r$ with $a_1\geq1$, 
		we define ${\mm{Li}}_{\mathbf{u}_r}^{\textup{\foreignlanguage{russian}{\cyrsh}}}(\mathbf{z}_r;\mathbf{a}_r)$ by
		\begin{align}\label{deffivepoly}
			{\mm{Li}}_{\mathbf{u}_r}^{\textup{\foreignlanguage{russian}{\cyrsh}}}(\mathbf{z}_r;\mathbf{a}_r)=\sum_{\substack{l_j=a_j\\j=1,\ldots,r}}^{\infty}\frac{z_1^{l_1}\cdots z_r^{l_r}}{l_1^{u_1}\cdots(l_1+\cdots+l_r)^{u_r}}.
		\end{align}
		Also, for $\mathbf{k}_r\in\Z^r, \mathbf{n}_r\in\Z_{\geq0}^r, \sigma\in\mathfrak{S}_r, \mathbf{a}_r\in\Z_{\geq0}^r$ with 
		$a_1,a_{\sigma^{-1}(1)}\geq1$ and $\mathbf{b}_r=(b_1,\ldots,b_r)\in\Z_{\geq0}^r$ with $b_1,b_{\sigma^{-1}(1)}\geq1$, we define $B_{\mathbf{n}_r}^{(-\mathbf{k}_r)}(\sigma;\mathbf{a}_r;\mathbf{b}_r)$ by
		\begin{align}\label{deffiveber}
			\begin{split}
				&\prod_{j=1}^{r}e^{(a_j-1)(t_j+\cdots+t_r)}\frac{{\mm{Li}}_{\mathbf{k}_r}^{\cy}(1-e^{-t_{\sigma(1)}-t_{\sigma(1)+1}-\cdots-t_r},\ldots,1-e^{-t_{\sigma(r)}-t_{\sigma(r)+1}-\cdots-t_r};\mathbf{b}_r)}{(1-e^{-t_{\sigma(1)}-
						t_{\sigma(1)+1}-\cdots-t_r})^{b_1}\cdots(1-e^{-t_{\sigma(r)}-t_{\sigma(r)+1}-\cdots-t_r})^{b_r}}\\
				=&\sum_{m_1,\ldots,m_r=0}^{\infty} B_{\mathbf{m}_r}^{(\mathbf{k}_r)}(\sigma;\mathbf{a}_r;\mathbf{b}_r)\frac{t_1^{m_1}}{m_1!}\cdots\frac{t_r^{m_r}}{m_r!}.
			\end{split}
		\end{align}
		Furthermore, for $\mathbf{u}_r, \mathbf{s}_r\in\Co^r$ with $\mm{Re}(u_j), \mm{Re}(s_j)>0$, $ \sigma\in\mathfrak{S}_r, \mathbf{a}_r\in\Z_{\geq0}^r$ with 
		$a_1,a_{\sigma^{-1}(1)}\geq1$ and $\mathbf{b}_r\in\Z_{\geq0}^r$ with $b_1,b_{\sigma^{-1}(1)}\geq1$ and $a_{\sigma(j)}+b_j\geq1$, we define $\eta(\mathbf{u}_r;\mathbf{s}_r;\sigma;\mathbf{a}_r;\mathbf{b}_r)$ by
		\begin{align}\label{intfiveeta}
			\begin{split}
				&\eta(\mathbf{u}_r;\mathbf{s}_r;\sigma;\mathbf{a}_r;\mathbf{b}_r)\\
				=&\prod_{j=1}^{r}\frac{1}{\Gamma(s_j)}\int_{(0,\infty)^r}^{}\prod_{j=1}^{r}t_j^{s_j-1}e^{(1-a_j)(t_j+\cdots+t_r)}
				\frac{{\mm{Li}}_{\mathbf{u}_r}^{\cy}(1-e^{t_{\sigma(1)}+t_{\sigma(1)+1}+\cdots+t_r},\ldots,1-e^{t_{\sigma(r)}+t_{\sigma(r)+1}+\cdots+t_r};\mathbf{b}_r)}{(1-e^{t_{\sigma(1)}+t_{\sigma(1)+1}+\cdots+t_r})^{b_1}\cdots(1-e^{t_{\sigma(r)}+t_{\sigma(r)+1}+\cdots+t_r})^{b_r}}
				\prod_{j=1}^{r}dt_j.
			\end{split}
		\end{align}
	\end{definition} 
	The Bernoulli numbers $B_{\mathbf{n}_r}^{(-\mathbf{k}_r)}(\sigma;\mathbf{a}_r;\mathbf{b}_r)$ are one of the generalizations of $B_{\mathbf{n}_r}^{(-\mathbf{k}_r),(r)}$. Also, the Arakawa-Kaneko type $\eta$ function $\eta(\mathbf{u}_r;\mathbf{s}_r;\sigma;\mathbf{a}_r;\mathbf{b}_r)$ are one of the generalizations of $\eta(\mathbf{u}_r;\mathbf{s}_r)$ (for details, see Section 4). The aim of this paper is to show that the function $\eta(\mathbf{u}_r;\mathbf{s}_r;\sigma;\mathbf{a}_r;\mathbf{b}_r)$ satisfies a certain duality formula. Also, we show that we can write the special values of $\eta(\mathbf{u}_r;\mathbf{s}_r;\sigma;\mathbf{a}_r;\mathbf{b}_r)$ at positive integers 
	in terms of multiple zeta values, and at non-positive integers in terms of $B_{\mathbf{n}_r}^{(-\mathbf{k}_r)}(\sigma;\mathbf{a}_r;\mathbf{b}_r)$. 
	
	The paper is organized as follows. In Section 2, we review some notations and the known results. In Section 3, we show analytic continuations of ${\mm{Li}}_{\mathbf{u}_r}^{\textup{\foreignlanguage{russian}{\cyrsh}}}(\mathbf{z}_r;\mathbf{a}_r)$ and $\eta(\mathbf{u}_r;\mathbf{s}_r;\sigma;\mathbf{a}_r;\mathbf{b}_r)$. Also, we show that the function $\eta(\mathbf{u}_r;\mathbf{s}_r;\sigma;\mathbf{a}_r;\mathbf{b}_r)$ satisfies a certain duality formula.
	In Section 4, we consider some special values of $\eta(\mathbf{u}_r;\mathbf{s}_r;\sigma;\mathbf{a}_r;\mathbf{b}_r)$.
	
	\section{Preliminaries}
	
	In this section, we recall some notations and the known results. First of all, we review the  properties $\eta(\mathbf{u}_r;\mathbf{s}_r)$. As we state in Section 1, the function $\eta(\mathbf{k}_r;s)$ is first defined by Kaneko and Tsumura \cite{KT1}. In the article, they gave the following conjecture, which was already proved.
	
	\begin{theorem}[{cf. \cite[p. 37]{KT1}}]\label{etaposi}
		For $k, n\in\Z_{\geq1}$,
		\begin{align}\label{oneetadual}
			\eta(k;n)=\eta(n;k)
		\end{align}
		holds.
	\end{theorem} 
	
	This theorem was first proved by Yamamoto \cite{Ya}, who showed the more general case. He showed the following results.
	\begin{theorem}[{\cite[Lemma 2.1]{Ya}}]
		For $\mathbf{u}_r, \mathbf{z}_r\in\Co^r$ with $|z_j|<1~(j=1,\ldots,r)$ and sufficiently small $\varepsilon\in\R_{>0}$, the multiple polylogarithm of $\cy$-type
		\[
		{\mm{Li}}_{\mathbf{u}_r}^{\cy}(\mathbf{z}_r)=\sum_{l_1,\ldots,l_r=1}^{\infty}\frac{z_1^{l_1}\cdots z_r^{l_r}}{l_1^{u_1}\cdots(l_1+\cdots+l_r)^{u_r}}
		\]
		has the integral expression
		\begin{align}\label{analy}
			{\mm{Li}}_{\mathbf{u}_r}^{\cy}(\mathbf{z}_r)=\prod_{j=1}^{r}\frac{\Gamma(1-u_j)}{2\pi ie^{\pi iu_j}}\int_{(\mathcal{C}_\varepsilon)^r}^{} \prod_{j=1}^{r}\frac{x_j^{u_j-1}z_j}{e^{x_j+\cdots+x_r}-z_j}dx_j.
		\end{align}
		Here, $\mathcal{C}_\varepsilon$ denotes the contour which goes from $+\infty$ to $\varepsilon$ along the real axis, goes round counterclockwise along the circle around the origin of radius $\varepsilon$~$($let $C(0;\varepsilon)$ be this circle$)$, and then goes back to $+\infty$ along the real axis. By \eqref{analy}, ${\mm{Li}}_{\mathbf{u}_r}^{\cy}(\mathbf{z}_r)$ can be continued analytically to the region
		\[
		(\mathbf{u}_r, \mathbf{z}_r)\in\Co^r\times(\Co\setminus\R_{\geq1})^r.
		\]
	\end{theorem}
	
	\begin{theorem}[{\cite[Definition 2.3, Theorem 2.5]{Ya}}]
		For $\mathbf{u}_r, \mathbf{s}_r\in\Co^r$ with $\mm{Re}(s_j)>0$, the function
		\begin{align}\label{yamamotoeta}
			\eta(\mathbf{u}_r;\mathbf{s}_r)=\prod_{j=1}^{r}\frac{1}{\Gamma(s_j)}\int_{0}^{\infty} t_1^{s_1-1}\cdots t_r^{s_r-1}\frac{{\mm{Li}}_{\mathbf{u}_r}^{\cy}(1-e^{t_1+\cdots+t_r},\ldots,1-e^{t_r})}{(1-e^{t_1+\cdots+t_r})\cdots(1-e^{t_r})}dt_1\cdots dt_r
		\end{align}
		is defined and has the integral expression
		\begin{align}\label{analyyameta}
			\eta(\mathbf{u}_r;\mathbf{s}_r)=\prod_{j=1}^{r}\frac{\Gamma(1-u_j)\Gamma(1-s_j)}{(2\pi i)^2e^{\pi i(u_j+s_j)}}\int_{(\mathcal{C}_{\varepsilon})^{2r}}^{} \prod_{j=1}^{r}
			\frac{x_j^{u_j-1}t_j^{s_j-1}}{e^{x_j+\cdots+x_r}+e^{t_j+\cdots+t_r}-1}dx_jdt_j.
		\end{align}
		By \eqref{analyyameta}, $\eta(\mathbf{u}_r;\mathbf{s}_r)$ can be continued analytically to the region
		\[
		(\mathbf{u}_r,\mathbf{s}_r)\in\Co^{2r}.
		\]
		Furthermore, $\eta(\mathbf{u}_r;\mathbf{s}_r)$ satisfies the duality formula
		\[
		\eta(\mathbf{u}_r;\mathbf{s}_r)=\eta(\mathbf{s}_r;\mathbf{u}_r).
		\]
	\end{theorem}
	
	To give the analytic continuation of $\eta(\mathbf{u}_r;\mathbf{s}_r)$, the following lemma is essential. For sufficiently small $\varepsilon\in\R_{>0}$, put $D_{\varepsilon}=\{\alpha\in\Co| -\varepsilon\leq\mm{Im}(\alpha)\leq\varepsilon, -\varepsilon\leq\mm{Re}(\alpha)\}$.
	
	\begin{lemma}[{\cite[Lemma 2.4]{Ya}}]\label{yamfivelemma}
		For $x, t\in D_{\varepsilon}$ and sufficiently small $\varepsilon\in\R_{>0}$,
		\begin{align}\label{4r}
			\left|\frac{e^{\frac{1}{2}x}e^{\frac{1}{2}t}}{e^x+e^t-1}\right|
		\end{align}
		is bounded.
	\end{lemma}
	
	\begin{remark}
		Kawasaki and Ohno gave a combinatorial proof of Theorem \ref{etaposi} (for details, see\cite{KO}).
	\end{remark}
	
	As analogues of \eqref{mulbber} and \eqref{kteta}, Ito defined multi-indexed poly-Bernoulli numbers of $C$-type and related Arakawa-Kaneko type zeta functions as follows.
	
	\begin{definition}[{\cite[Definition 1]{It2}, \cite[Definition 2.1]{It3}}]
		For $\mathbf{u}_r\in\Co^r$ and $d\in\{1,\ldots,r\}$, multi-indexed poly-Bernoulli numbers of $C$-type $C_{\mathbf{n}_r}^{(\mathbf{u}_r),(d)}$ are defined by 
		\begin{align}\label{itocber}
			\frac{{\mm{Li}}_{\mathbf{u}_r}^{\cy}(1-e^{-t_1-\cdots-t_r},\ldots,1-e^{-t_r})}{(e^{t_1+\cdots+t_r}-1)\cdots(e^{t_d+\cdots+t_r}-1)}=\sum_{m_1,\ldots,m_r=0}^{\infty} 
			C_{\mathbf{m}_r}^{(\mathbf{u}_r),(d)}\frac{t_1^{m_1}}{m_1!}\cdots\frac{t_r^{m_r}}{m_r!}.
		\end{align}
		Also, for $\mathbf{k}_r\in\Z_{\geq1}^r, -\mathbf{l}_r=(-l_1,\ldots,-l_r)\in\Z_{\leq0}^r$ with $l_1\geq1$ and $d\in\{1,\ldots,r\}$, related Arakawa-Kaneko type zeta functions 
		$\xi(\mathbf{k}_r;\mathbf{s}_r;d)$ and $\widetilde{\xi}(-\mathbf{l}_r;\mathbf{s}_r;d)$ are defined by
		\begin{align*}
			\xi(\mathbf{k}_r;\mathbf{s}_r;d)&=\prod_{j=1}^{r}\frac{1}{\Gamma(s_j)}\int_{0}^{\infty} t_1^{s_1-1}\cdots t_r^{s_r-1}\frac{{\mm{Li}}_{\mathbf{k}_r}^{\cy}(1-e^{-t_1-\cdots-t_r},\ldots,1-e^{-t_r})}{(e^{t_1+\cdots+t_r}-1)\cdots(e^{t_d+\cdots+t_r}-1)}dt_1\cdots dt_r,\\
			\widetilde{\xi}(-\mathbf{l}_r;\mathbf{s}_r;d)&=\prod_{j=1}^{r}\frac{1}{\Gamma(s_j)}\int_{0}^{\infty} t_1^{s_1-1}\cdots t_r^{s_r-1}\frac{{\mm{Li}}_{-\mathbf{l}_r}^{\cy}(1-e^{t_1+\cdots+t_r},\ldots,1-e^{t_r})}{(e^{-t_1-\cdots-t_r}-1)\cdots(e^{-t_d-\cdots-t_r}-1)}dt_1\cdots dt_r.
		\end{align*}
	\end{definition}
	Ito showed the following relations.
	\begin{theorem}[{\cite[Theorem 1]{It2}, \cite[Theorem 3.1]{It3}}]
		For $\mathbf{k}_r\in\Z_{\geq1}^r, \mathbf{l}_r\in\Z_{\geq0}^r$ with $l_1\geq1$, $\mathbf{n}_r\in\Z_{\geq0}^r$ and $d\in\{1,\ldots,r\}$, we have
		\begin{align*}
			\xi(\mathbf{k}_r;-\mathbf{n}_r;d)&=(-1)^{n_1+\cdots+n_r}C_{\mathbf{n}_r}^{(\mathbf{k}_r),(d)},\\
			\widetilde{\xi}(-\mathbf{l}_r;-\mathbf{n}_r;d)&=C_{\mathbf{n}_r}^{(-\mathbf{l}_r),(d)}.
		\end{align*}
	\end{theorem}
	\begin{theorem}[{\cite[Theorem 2]{It2}}]
		For $k_1,\ldots,k_r\in\Z_{\geq0}$, we have
		\begin{align*}
			\widetilde{\xi}(-k_1-1,-k_2,\ldots,-k_r;s_1,\ldots,s_r;r)&=\sum_{a=1}^{r}\sum_{1=b_1<\cdots<b_a\leq r}^{}C_{b(k_1,\ldots,k_r)}^{b(-s_1-1,\ldots,-s_r),(1)},
		\end{align*}
		where $b(k_1,\ldots,k_r)=(k_1+\cdots+k_{b_2-1},k_{b_2}+\cdots+k_{b_3-1},\ldots,k_{b_{a-1}}+\cdots+k_{r})$. In particular, for $n_1,\ldots,n_r\in\Z_{\geq0}$, we have
		\begin{align}\label{dualc}
			C_{n_1,\ldots,n_r}^{(-k_1-1,\ldots,-k_r),(r)}=\sum_{a=1}^{r}\sum_{1=b_1<\cdots<b_a\leq r}^{}C_{b(k_1,\ldots,k_r)}^{b(-n_1-1,\ldots,-n_r),(1)},
		\end{align}
		which is a generalization of that for poly-Bernoulli numbers $\cber{n}{(-k)}$.
	\end{theorem}
	\begin{theorem}[{\cite[Theorem C.12]{It1}, {\cite[Theorem 4.3]{It3}}}]\label{mzvfunc}
		For $\mathbf{k}_r, \mathbf{n}_r\in\Z_{\geq1}$ and $d\in\{1,\ldots,r\}$, we have
		\begin{align*}
			\xi(\mathbf{k}_r;\mathbf{n}_r;d), \eta(\mathbf{k}_r;\mathbf{n}_r)\in\mathcal{Z}_{k_1+\cdots+k_r+n_1+\cdots+n_r}.
		\end{align*}
	\end{theorem}
	To show Theorem \ref{mzvfunc}, we review some properties of the hyperlogarithm. Note that, though we use the same terminology ``hyperlogarithm'', the following definition and results obtained by Ito are generalizations of known results.
	\begin{definition}[{cf. \cite{Lap}}]
		Let $a_0\in\R$, $a_{n+1}$ be a variable with $a_0<a_{n+1}$ and $a_1,\ldots,a_n$ be variables with $a_j\in\Co\setminus(a_0,a_{n+1})$ for each point, $a_1\neq a_0$ and $a_n\neq a_{n+1}$. For them, the hyperlogarithm is defined by
		\[
		I(a_0;a_1,\ldots,a_n;a_{n+1})=\int_{a_0<t_1<\cdots<t_n<a_{n+1}}^{}\prod_{j=1}^{n}\frac{dt_j}{t_j-a_j}.
		\]
	\end{definition}
	
	We can show the following lemma by definition.
	
	\begin{lemma}\cite[Lemma 4.4]{It3} Let the notation be the same as above. For $c\in\R_{>0}$, we have
		\begin{align*}
			I(a_0;a_1,\ldots,a_n;a_{n+1})=I(ca_0;ca_1,\ldots,ca_n;ca_{n+1}).
		\end{align*}
		Also, for $a_0<x, a_1,\ldots,a_n\in\Co\setminus(a_0, x)$ with $a_1\neq a_0, a_n\neq x$ and $b\in\Co\setminus (a_0,x]$, we have
		\begin{align*}
			\int_{a_0}^{x}\frac{1}{y-b}I(a_0;a_1,\ldots,a_n;y) dy =I(a_0;a_1,\ldots,a_n,b;x).
		\end{align*}
	\end{lemma}
	
	Note that, for $\mathbf{k}_r\in\Z_{\geq1}^r$ and $\mathbf{z}_r\in\Co^r$ with $|z_j|<1~(j=1,\ldots,r)$, we have
	\begin{align}
		{\mm{Li}}_{\mathbf{k}_r}^{\cy}(\mathbf{z}_r)=(-1)^rI(0;z_1^{-1},\overbrace{0,\ldots,0}^{k_1-1},\ldots,z_r^{-1},\overbrace{0,\ldots,0}^{k_r-1};1).\label{polyhyp}
	\end{align}
	To obtain Theorem \ref{mzvfunc}, the following lemmas are essential.
	\begin{lemma}[{\cite[Lemma 4.5]{It3}, cf. \cite[Theorem 2.1]{Go}, \cite[Lemma 2.2]{HIST}, \cite[Lemma 3.3.30]{Pan}}]\label{itlem1}
		For $a_0\in\R$, $a_{n+1}=a_{n+1}(x)$ with $a_0<a_{n+1}(x)$ and $a_1=a_1(x),\ldots,a_n=a_n(x)$ with $a_j(x)\in\Co\setminus(a_0,a_{n+1}(x))$ for each point $x$, $a_1\neq a_0, a_n\neq a_{n+1}$, we have
		\[
		\pdv{}{x}I(a_0;a_1,\ldots,a_n;a_{n+1})=\sum_{j=1}^{n} (\epsilon_j(\mathbf{a})-\epsilon_{j-1}(\mathbf{a}))I(a_0;a_1,\ldots,\widehat{a_j},\ldots,a_n;a_{n+1}).
		\]
		Here, for $j=0,\ldots,n$ and $\mathbf{a}=(a_1,\ldots,a_n)$,
		\begin{align*}
			\epsilon_j(\mathbf{a})=\begin{cases}
				0&~(a_{j+1}=a_j),\\
				\frac{\pdv{}{x}(a_{j+1}-a_j)}{a_{j+1}-a_j}&~(a_{j+1}\neq a_j),
			\end{cases}
		\end{align*}
		and $(a_1,\ldots,\widehat{a_j},\ldots,a_n)=(a_1,\ldots,a_{j-1},a_{j+1},\ldots,a_n)$.
	\end{lemma}
	
	By Lemma \ref{itlem1}, we can transform hyperlogarithms. For example, consider transforming the hyperlogarithm
	\[
	I(0;x_2^{-1}, x_1^{-1};1).
	\]
	Since
	\[
	\pdv{}{x_2}I(0;x_2^{-1}, x_1^{-1};1)=\frac{1}{x_2-x_1}I(0;1;x_1)+\left(\frac{1}{x_2}-\frac{1}{x_2-x_1}\right)I(0;1;x_2),
	\]
	we obtain
	\[
	I(0;x_2^{-1}, x_1^{-1};1)=I(0;x_1;x_2)I(0;1;x_1)+I(0;1,0;x_2)-I(0;1,x_1;x_2).
	\]
	Ito called this transforming process {\it variable removing}, and we use the same terminology.
	
	\begin{lemma}[{\cite[Lemma 4.6]{It3}}]
		For $x\in\R$ and $x_1,\ldots,x_r\in\Co$ with $0<x<|x_j|<1~(j=1,\ldots,r)$, define $V(x_1,\ldots,x_r;x)$ as the $\Q$-linear space by
		\begin{align*}
			V(x_1,\ldots,x_r;x)=\left\langle{I(0;a_1,\ldots,a_n;1)I(0;b_1,\ldots,b_l;x)} ~\middle|~ {\Large\substack{a_j\in\{0,1,x_j^{-1}\},~ a_1\neq0,~ a_n\neq1,\\ b_j\in\{0,1,x_j\},~ b_1\neq0}}\right\rangle_{\Q}.
		\end{align*}
		Then, for $a_1,\ldots,a_n\in\{0,1,x^{-1},x_1^{-1},\ldots,x_r^{-1}\}$ with $a_1\neq0, a_1\neq1$, we have
		\begin{align*}
			I(0;a_1,\ldots,a_n;1)\in V(x_1,\ldots,x_r;x).
		\end{align*}
		Also, for $b_1,\ldots,b_l\in\{0,1,x,x_1,\ldots,x_r\}$ with $b_1\neq0, b_l\neq x$, we have
		\begin{align*}
			I(0;b_1,\ldots,b_l;x)\in V(x_1,\ldots,x_r;x).
		\end{align*}
	\end{lemma}
	
	\begin{lemma}[{\cite[Proposition 4.8]{It3}}]\label{itlem2}
		Let $n,l\in\Z_{\geq0}, a_j\in\{0,1,x_1^{-1},\ldots,x_r^{-1}\}~(j=1,\ldots,n)$ with $a_1\neq0, a_n\neq1$, $b_j\in\{0,1,x_1,\ldots,x_r\}~(j=1,\ldots,l)$ with $b_1\neq0, b_l\neq x_r$, and $c_j\in\{0,1\}~(j=2,\ldots,r)$. Suppose that for all $j\in\{1,\ldots,r\}$, there exists some $v\in\{1,\ldots,n\}$ such that $x_j^{-1}=a_v$ or $x_j=b_v$. Then we have
		\[
		\int_{0<x_r<\cdots<x_1<1}^{} I(0;a_1,\ldots,a_n;1)I(0;b_1,\ldots,b_l;x_r)\frac{dx_1}{x_1}\prod_{j=2}^{r}\frac{dx_j}{x_j-c_j}\in\mathcal{Z}_{n+l+r}.
		\]
	\end{lemma}
	These lemmas are also needed in Section 3. 
	
	\section{Analytic continuations of ${\mm{Li}}_{\mathbf{u}_r}^{\textup{\foreignlanguage{russian}{\cyrsh}}}(\mathbf{z}_r;\mathbf{a}_r)$ and $\eta(\mathbf{u}_r;\mathbf{s}_r;\sigma;\mathbf{a}_r;\mathbf{b}_r)$}
	
	In this section, we show analytic continuations of ${\mm{Li}}_{\mathbf{u}_r}^{\textup{\foreignlanguage{russian}{\cyrsh}}}(\mathbf{z}_r;\mathbf{a}_r)$ and $\eta(\mathbf{u}_r;\mathbf{s}_r;\sigma;\mathbf{a}_r;\mathbf{b}_r)$.

	\begin{lemma}[{\cite[p. 2]{Ya}}]\label{zerodenai}
		For each $\mathbf{z}_r\in(\Co\setminus\R_{\geq1})^r$, there exists a neighborhood $K$ of $\mathbf{z}_r$ and $\varepsilon_0\in\R_{>0}$ such that $e^{x_j+\cdots+x_r}-z_j^{\prime}\neq0$
		for $j=1,\ldots,r$ whenever $\mathbf{z}_r\in K, 0<\varepsilon<\varepsilon_0$ and $x_1,\ldots, x_r\in\mathcal{C}_{\varepsilon}$. 
	\end{lemma}
	
	\begin{lemma}\label{symfivelemma}
		For $\mathbf{u}_r, \mathbf{z}_r\in\Co^r$ with $|z_j|<1$, $\mathbf{a}_r\in\Z_{\geq0}^{r}$ with $a_1=1$ and sufficiently small $\varepsilon\in\R_{>0}$, we have
		\begin{align}\label{contourintli}
			{\mm{Li}}_{\mathbf{u}_r}^{\textup{\foreignlanguage{russian}{\cyrsh}}}(\mathbf{z}_r;\mathbf{a}_r)=\prod_{j=1}^{r}\frac{1}{(e^{2\pi iu_j}-1)\Gamma(u_j)}\int_{(\mathcal{C}_\varepsilon)^r}\prod_{j=1}^{r}\frac{x_j^{u_j-1}e^{(1-a_j)(x_j+\cdots+x_r)}z_j^{a_j}}{e^{x_j+\cdots+x_r}-z_j}dx_j.
		\end{align}
		In particular, ${\mm{Li}}_{\mathbf{u}_r}^{\cy}(\mathbf{z}_r;\mathbf{a}_r)$ can be continued analytically to the region $(\mathbf{u}_r,\mathbf{z}_r)\in\Co^r\times(\Co\setminus\R_{\geq1})^r$ .
	\end{lemma}

	\begin{proof}
		The argument in this proof is similar to that in \cite[Lemma 2.1]{Ya}. For $j=1,\ldots, r$, fix $z_j\in\Co\setminus\R_{\geq1}$ and suppose $\mm{Re}(u_j)>1$. By the series expression, we have
		\begin{align*}
			{\mm{Li}}_{\mathbf{u}_r}^{\cy}(\mathbf{z}_r;\mathbf{a})&=\sum_{\substack{l_j=a_j\\j=1,\ldots,r}}^{\infty}\frac{z_1^{l_1}\cdots z_r^{l_r}}{l_1^{u_1}\cdots(l_1+\cdots+l_r)^{u_r}}\\
			&=\prod_{j=1}^{r}\frac{1}{\Gamma(u_j)}\int_{(0,\infty)^r}^{}\prod_{j=1}^{r}\frac{x_j^{u_j-1}e^{(1-a_j)(x_j+\cdots+x_r)}z_j^{a_j}}{e^{x_j+\cdots+x_r}-z_j}dx_1\cdots dx_r.
		\end{align*}
		Hence we obtain the analytic continuation of ${\mm{Li}}_{\mathbf{u}_r}^{\cy}(\mathbf{z}_r;\mathbf{a}_r)$ to the region
		\[
		(\mathbf{u}_r,\mathbf{z}_r)\in\{u\in\Co \mid \mm{Re}(u)>1\}^r\times(\Co\setminus\R_{\geq1})^r.
		\] 
		Here, we define $H(\mathbf{u}_r)$ by
		\begin{align*}
			H(\mathbf{u}_r)&=\int_{(\mathcal{C}_{\varepsilon})^r}^{}\prod_{j=1}^{r}\frac{x_j^{u_j-1}e^{(1-a_j)(x_j+\cdots+x_r)}z_j^{a_j}}{e^{x_j+\cdots+x_r}-z_j}dx_j\\
			&=\sum_{\substack{I_1,\ldots,I_r\\=C(0;\varepsilon)\mm{or}(\varepsilon,\infty)}}^{}\prod_{j:I_j=(\varepsilon,\infty)}^{}(e^{2\pi iu_j}-1)\int_{I_1\times\cdots\times I_r}^{} \prod_{j=1}^{r}\frac{x_j^{u_j-1}e^{(1-a_j)(x_j+\cdots+x_r)}z_j^{a_j}}{e^{x_j+\cdots+x_r}-z_j}dx_j,
		\end{align*}
		where the product $\prod_{j:I_j=(\varepsilon,\infty)}^{}$ runs over all $j$ with $I_j=(\varepsilon,\infty)$.
		By Lemma \ref{zerodenai}, we can assume that $e^{x_j+\cdots+x_r}-z_j\neq0$. Furthermore, for $x\in\{\alpha\in\Co \mid -\varepsilon\leq\mm{Im}(\alpha)\leq\varepsilon, -\varepsilon\leq\mm{Re}(\alpha)\}$, we have
		\begin{align*}
			\left|\frac{1}{e^x-z}\right|&\ll e^{-\frac{1}{2}\mm{Re}(x)},\\
			\frac{e^x}{e^x-z}&=O(1).
		\end{align*}
		Therefore, except for the case $I_1=\cdots=I_r=(\varepsilon,\infty)$, we have
		\begin{align*}
			\left|\int_{I_1\times\cdots\times I_r}^{} \prod_{j=1}^{r}\frac{x_j^{u_j-1}e^{(1-a_j)(x_j+\cdots+x_r)}z_j^{a_j}}{e^{x_j+\cdots+x_r}-z_j}dx_j\right|
			&\leq M_1\int_{I_1\times\cdots\times I_r}^{} \prod_{j=1}^{r}\left|\frac{x_j^{u_j-1}}{e^{\frac{1}{2}x_j}}\right|\prod_{j:I_j=C(0;\varepsilon)}^{}|dx_j|\prod_{l:I_l=(\varepsilon,\infty)}^{}dx_l\\
			&\leq M_2\prod_{I_j=C(0;\varepsilon)}^{}\varepsilon^{\mm{Re}(u_j)}\\
			&\rightarrow 0~(\varepsilon\rightarrow 0).
		\end{align*}
		Here, $M_2$ is the constant which depends on all $u_j$. Hence we obtain
		\[
		H(\mathbf{u}_r)=\prod_{j=1}^{r}(e^{2\pi iu_j}-1)\int_{(0,\infty)^r}^{}\frac{x_j^{u_j-1}e^{(1-a_j)(x_j+\cdots+x_r)}z_j^{a_j}}{e^{x_j+\cdots+x_r}-z_j}dx_j.
		\]
		On the other hand, since
		\begin{align*}
			&\left|\int_{(\mathcal{C}_{\varepsilon})^r}^{}\frac{x_j^{u_j-1}e^{(1-a_j)(x_j+\cdots+x_r)}z_j^{a_j}}{e^{x_j+\cdots+x_r}-z_j}dx_j\right|\\
			\leq&M_1\sum_{\substack{I_1,\ldots,I_r\\=C(0;\varepsilon)\mm{or}(\varepsilon,\infty)}}^{}\prod_{j:I_j=(\varepsilon,\infty)}^{}\int_{I_j}^{}\left|\frac{x_j^{u_j-1}}{e^{\frac{1}{2}x_j}}\right|dx_j\prod_{l:I_l=C(0;\varepsilon)}^{}\int_{I_l}^{}\left|\frac{x_l^{u_l-1}}{e^{\frac{1}{2}x_l}}\right||dx_l|,
		\end{align*}
		$H(\mathbf{u}_r)$ converges absolutely. Hence, for $\mathbf{u}_r\in\Co^r$, we have
		\begin{align*}
			{\mm{Li}}_{\mathbf{u}_r}^{\cy}(\mathbf{z}_r;\mathbf{a})&=\prod_{j=1}^{r}\frac{1}{(e^{2\pi iu_j}-1)\Gamma(u_j)}\int_{(\mathcal{C}_{\varepsilon})^r}^{}\prod_{j=1}^{r}\frac{x_j^{u_j-1}e^{(1-a_j)(x_j+\cdots+x_r)}z_j^{a_j}}{e^{x_j+\cdots+x_r}-z_j}dx_j,
		\end{align*}
		and this gives the analytic continuation of ${\mm{Li}}_{\mathbf{u}_r}^{\cy}(\mathbf{z}_r;\mathbf{a})$ to the region
		\[
		(\mathbf{u}_r, \mathbf{z}_r)\in\Co^r\times(\Co\setminus\R_{\geq1})^r.
		\] 
	\end{proof}
	
	\begin{remark}
		In \cite[Theorem 3.17]{Kom1}, Komori defined 
		\begin{align*}
			\zeta(\boldsymbol{\xi}, \mathbf{d}, \mathbf{b}, \mathbf{s})&=\prod_{j=1}^{N}\frac{1}{\Gamma(s_j)}\prod_{t\in S}^{}\frac{1}{e^{2\pi it(\mathbf{s})}-1}\\
			&\five \times\int_{\widehat{\Sigma}}^{} \frac{e^{(b_{11}+\cdots+b_{1R}-d_1)z_1}\cdots e^{(b_{N1}+\cdots+b_{NR}-d_N)z_N}z_1^{s_1-1}\cdots z_N^{s_N-1}}{(e^{z_1b_{11}+\cdots+z_Nb_{N1}}-e^{\xi_1})\cdots(e^{z_1b_{1R}+\cdots+z_Nb_{NR}-e^{\xi_r}})} dz_1\wedge\cdots \wedge dz_r
		\end{align*}
		for $N, R\in\Z_{\geq1}, \boldsymbol{\xi}=(\xi_1,\ldots,\xi_R)\in(\Co/2\pi i\Z)^R, \mathbf{d}=(d_1,\ldots,d_N), \mathbf{s}=(s_1,\ldots,s_N)\in\Co^N, \mathbf{b}=(b_{nr})_{1\leq n \leq N, 1\leq r \leq R}\in\Co^{N\times R}, S:$ a set of linear functionals on $\Co^N$ and $\widehat{\Sigma}:$ a
		union of smooth surfaces. Also, he gives the analytic continuation of $\zeta(\boldsymbol{\xi}, \mathbf{a}, \mathbf{b}, \mathbf{s})$. Lemma \ref{symfivelemma} is the case $N=R=r, b_{ij}=1, 
		d_j=r-\sum_{l=1}^{j}(1-a_l)$ and 
		$S=\{t_j:\Co^N \rightarrow \Co^N | t_j(s_1,\ldots,s_r)=s_j\}$.
	\end{remark}
	
	\begin{proposition}\label{propcontourinteta}
		For $\mathbf{u}_r, \mathbf{s}_r\in\Co^r, \sigma\in\mathfrak{S}_r, \mathbf{a}_r\in\Z_{\geq0}^r$ with 
		$a_1,a_{\sigma^{-1}(1)}\geq1$ and $\mathbf{b}_r\in\Z_{\geq0}^r$ with $b_1,b_{\sigma^{-1}(1)}\geq1$ and $a_{\sigma(j)}+b_j\geq1$, the function $\eta(\mathbf{u}_r;\mathbf{s}_r;\sigma;\mathbf{a}_r;\mathbf{b}_r)$ has the integral expression
		\begin{align}\label{contourinteta}
			\begin{split}
				&\eta(\mathbf{u}_r;\mathbf{s}_r;\sigma;\mathbf{a}_r;\mathbf{b}_r)\\
				=&\prod_{j=1}^{r}\frac{1}{(e^{2\pi iu_j}-1)(e^{2\pi is_j}-1)\Gamma(u_j)\Gamma(s_j)}\int_{\mathcal{C}^{2r}}^{}\prod_{j=1}^{r}t_j^{s_j-1}x_j^{u_j-1}
				\frac{e^{(1-a_{\sigma(j)})(t_{\sigma(j)}+t_{\sigma(j)+1}+\cdots+t_r)}e^{(1-b_j)(x_j+\cdots+x_r)}}{e^{t_{\sigma(j)}+t_{\sigma(j)+1}+\cdots+t_r}+e^{x_j+\cdots+x_r}-1} 
				dt_jdx_j.
			\end{split}
		\end{align}
		In particular, the function $\eta(\mathbf{u}_r;\mathbf{s}_r;\sigma;\mathbf{a}_r;\mathbf{b}_r)$ can be continued analytically to the region
		\[
		(\mathbf{u}_r,\mathbf{s}_r)\in\Co^{2r}.
		\]
	\end{proposition}
	
	To prove Proposition \ref{propcontourinteta}, we need the following lemma.
	
	\begin{lemma}[{cf. \cite[Lemma 2.4]{Ya}}]\label{fivelemma}
		For $x, t\in D_{\varepsilon}$ and sufficiently small $\varepsilon\in\R_{>0}$,
		\begin{align}\label{4r}
			\left|\frac{e^{\frac{4r-1}{4r}x}e^{\frac{1}{4r}t}}{e^x+e^t-1}\right|
		\end{align}
		can be bounded by a constant which does not depend on $\varepsilon$.
	\end{lemma}
	
	\begin{proof}
		Put $w=e^x-\frac{1}{2}, y=e^t-\frac{1}{2}, x=x_1+x_2i$ and $t=t_1+t_2i$. Then we have
		\[
		\left|\frac{e^{\frac{4r-1}{4r}x}e^{\frac{1}{4r}t}}{e^x+e^t-1}\right|=\frac{\mm{Re}(w)+\mm{Re}(y)}{|w+y|}\times\frac{|w|+|y|}{\mm{Re}(w)+\mm{Re}(y)}\times\frac{|w|^{\frac{4r-1}{4r}}|y|^{\frac{1}{4r}}}{|w|+|y|}\times\left|\frac{e^x}{w}\right|^{\frac{4r-1}{4r}}\times\left|\frac{e^t}{y}\right|^{\frac{1}{4r}}.
		\]
		The first term is bounded by $1$. For the forth and fifth term, since
		\begin{align*}
			\left|\frac{e^x}{e^x-\frac{1}{2}}\right|&=\left|\frac{1}{1-\frac{1}{2}(e^{-x_1}\cos{x_2}-ie^{-x_1}\sin{x_2})}\right|\\
			&=\frac{1}{\sqrt{(1-\frac{1}{2}e^{-x_1}\cos{x_2})^2+\frac{1}{4}e^{-2x_1}\sin^2{x_2}}}\\
			&\leq \frac{2}{|e^{-x_1}-2\cos{x_2}|},
		\end{align*}
		we obtain the bound $\frac{2}{\sqrt{3}-\sqrt{\frac{e+3}{2}}}$ if $\varepsilon<\frac{1}{2}\log{\frac{e+3}{2}}$. For the second term, if we take $\varepsilon$ satisfying $\varepsilon<\frac{1}{4}\log{3}$, we have
		\begin{align*}
			\frac{|w|+|y|}{\mm{Re}(w)+\mm{Re}(y)}&=\frac{|e^{x_1+x_2i}-\frac{1}{2}|+|e^{t_1+t_2i}-\frac{1}{2}|}{\mm{Re}(e^{x_1+x_2i}-\frac{1}{2})+\mm{Re}(e^{t_1+t_2i}-\frac{1}{2})}\\
			&\leq \frac{e^{x_1}+e^{t_1}+1}{\frac{\sqrt{3}}{2}\left(e^{x_1}+e^{t_1}\right)-1}\\
			&\leq \frac{1+\frac{\sqrt[4]{3}}{2}}{\frac{\sqrt{3}}{2}-\frac{\sqrt[4]{3}}{2}}.
		\end{align*}
		Hence the second term can be bounded by a constant which does not depend on $\varepsilon$. For the third term, since $\frac{|w|^{\frac{4r-1}{4r}}|y|^{\frac{1}{4r}}}{|w|+|y|}>0$, we have
		\begin{align*}
			\left(\frac{|w|^{\frac{4r-1}{4r}}|y|^{\frac{1}{4r}}}{|w|+|y|}\right)^{4r}&\leq\frac{|w|^{4r-1}|y|}{4r|w|^{4r-1}|y|}=\frac{1}{4r}<1.
		\end{align*}
		Therefore, if $\varepsilon<\frac{1}{4}\log{3}$, we can bound \eqref{4r} by a constant which does not depend on $\varepsilon$. 
	\end{proof}
	
	\begin{proof}[Proof of Proposition \ref{propcontourinteta}]
		By \eqref{contourintli}, for $\mm{Re}(s_j)>0~(j=1,\ldots, r)$, we have
		\begin{align*}
			&\eta(\mathbf{u}_r;\mathbf{s}_r;\sigma;\mathbf{a}_r;\mathbf{b}_r)\\
			=&\prod_{j=1}^{r}\frac{1}{(e^{2\pi iu_j}-1)\Gamma(u_j)\Gamma(s_j)}\int_{(0,\infty)^r}^{}\int_{(\mathcal{C}_\varepsilon)^r}^{}\prod_{j=1}^{r}t_j^{s_j-1}x_j^{u_j-1}\frac{e^{(1-a_{\sigma(j)})(t_{\sigma(j)}+t_{\sigma(j)+1}+\cdots+t_r)}e^{(1-b_j)(x_j+\cdots+x_r)}}{e^{t_{\sigma(j)}+t_{\sigma(j)+1}+\cdots+t_r}+e^{x_j+\cdots+x_r}-1}
			dt_jdx_j.
		\end{align*}
		Note that, when $a_{\sigma(j)}=0$, then $b_j\geq1$ and we have
		\[
		\frac{e^{(1-a_{\sigma(j)})(t_{\sigma(j)}+t_{\sigma(j)+1}+\cdots+t_r)}e^{(1-b_j)(x_j+\cdots+x_r)}}{e^{t_{\sigma(j)}+t_{\sigma(j)+1}+\cdots+t_r}+e^{x_j+\cdots+x_r}-1}=e^{(1-b_j)(x_j+\cdots+x_r)}\frac{
			e^{t_{\sigma(j)}+t_{\sigma(j)+1}+\cdots+t_r}}{e^{t_{\sigma(j)}+t_{\sigma(j)+1}\cdots+t_r}+e^{x_j+\cdots+x_r}-1}.
		\]
		Also, when $b_j=0$, then $a_{\sigma(j)}\geq1$ and we have
		\[
		\frac{e^{(1-a_{\sigma(j)})(t_{\sigma(j)}+t_{\sigma(j)+1}+\cdots+t_r)}e^{(1-b_j)(x_j+\cdots+x_r)}}{e^{t_{\sigma(j)}+t_{\sigma(j)+1}+\cdots+t_r}+e^{x_j+\cdots+x_r}-1}=e^{(1-a_{\sigma(j)})(t_{\sigma(j)}+t_{\sigma(j)+1}+\cdots+t_r)}
		\frac{e^{x_j+\cdots+x_r}}{e^{t_{\sigma(j)}+t_{\sigma(j)+1}\cdots+t_r}+e^{x_j+\cdots+x_r}-1}.
		\]
		Otherwise, we have
		\begin{align*}
			&\frac{e^{(1-a_{\sigma(j)})(t_{\sigma(j)}+t_{\sigma(j)+1}+\cdots+t_r)}e^{(1-b_j)(x_j+\cdots+x_r)}}{e^{t_{\sigma(j)}+t_{\sigma(j)+1}+\cdots+t_r}+e^{x_j+\cdots+x_r}-1}\\
			=&e^{(1-a_{\sigma(j)})(t_{\sigma(j)}+t_{\sigma(j)+1}+\cdots+t_r)}e^{(1-b_j)(x_j+\cdots+x_r)}\frac{1
			}{e^{t_{\sigma(j)}+t_{\sigma(j)+1}\cdots+t_r}+e^{x_j+\cdots+x_r}-1}.
		\end{align*}
		Therefore, by Lemma \ref{fivelemma} , we obtain
		\begin{align*}
			&\left|\prod_{j=1}^{r}t_j^{s_j-1}x_j^{u_j-1}\frac{e^{(1-a_{\sigma(j)})(t_{\sigma(j)}+t_{\sigma(j)+1}+\cdots+t_r)}e^{(1-b_j)(x_j+\cdots+x_r)}}{e^{t_{\sigma(j)}+t_{\sigma(j)+1}+\cdots+t_r}+e^{x_j+\cdots+x_r}-1} \right|\\
			\leq&M\left|\prod_{j=1}^{r}\frac{x_j^{u_j-1}t_j^{s_j-1}}{e^{\frac{1}{2}t_{j}}e^{\frac{1}{2}x_j}}\right|\left|\prod_{j=2}^{r}e^{\frac{1}{4r}(t_{j}+\cdots+t_r)}
			e^{\frac{1}{4r}(x_j+\cdots+x_r)}\right|\\
			\leq&M\left|\prod_{j=1}^{r}\frac{x_j^{u_j-1}t_j^{s_j-1}}{e^{\frac{1}{4}t_{j}}e^{\frac{1}{4}x_j}}\right|
		\end{align*}
		for a constant $M$. Hence we have
		\begin{align*}
			&\left|\int_{(0,\infty)^r}^{}\int_{(\mathcal{C}_\varepsilon)^r}^{}\prod_{j=1}^{r}t_j^{s_j-1}x_j^{u_j-1}\frac{e^{(1-a_j)(x_j+\cdots+x_r)}e^{-b_{\sigma(j)}(t_{\sigma(j)}+t_{\sigma(j)+1}+\cdots+t_r)}}{e^{x_j+\cdots+x_r}+e^{t_{\sigma(j)}+t_{\sigma(j)+1}+\cdots+t_r}-1} dt_jdx_j\right|\\
			\leq&M\sum_{\substack{I_1,\ldots,I_r\\=C(0;\varepsilon)~\mm{or}~(\varepsilon,\infty)}}^{}\prod_{l:I_l=(\varepsilon,\infty)}^{}(e^{2\pi iu_l}-1)\prod_{j=1}^{r}\int_{(0,\infty)}^{}\left|\frac{x_j^{u_j-1}}{e^{\frac{1}{4}x_j}}\right|dx_j\\
			&\five\times\prod_{j:I_j=(\varepsilon,\infty)}{}\int_{I_j}^{} \left|\frac{t_j^{s_j-1}}{e^{\frac{1}{4}t_j}}\right|dt_j\prod_{l:I_l=C(0;\varepsilon)}{}\int_{I_l}^{} \left|\frac{t_l^{s_l-1}}{e^{\frac{1}{4}t_l}}\right||dt_l|.
		\end{align*}
		Therefore, \eqref{contourinteta} converges absolutely. Moreover, except for the case $I_1=\cdots=I_r=(\varepsilon,\infty)$, it holds that, for $\mm{Re}(u_j)>0~(j=1,\ldots,r)$,
		\begin{align*}
			&\left|\int_{(0,\infty)^r}^{}\int_{I_1\times\cdots\times I_r}^{}\prod_{j=1}^{r}t_j^{s_j-1}x_j^{u_j-1}\frac{e^{(1-a_j)(x_j+\cdots+x_r)}e^{-b_{\sigma(j)}(t_{\sigma(j)}+t_{\sigma(j)+1}+\cdots+t_r)}}{e^{x_j+\cdots+x_r}+e^{t_{\sigma(j)}+t_{\sigma(j)+1}+\cdots+t_r}-1} dt_jdx_j\right|\\
			\leq& M_2\prod_{j:I_j=C(0;\varepsilon)}^{} \varepsilon^{\mm{Re}(u_j)}\rightarrow0~(\varepsilon\rightarrow0)
		\end{align*}
		with a constant $M_2$ depending on all $u_j$ and $s_j$. Hence we obtain
		\begin{align*}
			\eta(\mathbf{u}_r;\mathbf{s}_r;\sigma;\mathbf{a}_r;\mathbf{b}_r)&=\prod_{j=1}^{r}\frac{1}{\Gamma(u_j)\Gamma(s_j)}\int_{(0,\infty)^{2r}}\prod_{j=1}^{r}t_j^{s_j-1}x_j^{u_j-1}\frac{e^{(1-a_{\sigma(j)})(t_{\sigma(j)}+t_{\sigma(j)+1}+\cdots+t_r)}e^{(1-b_j)(x_j+\cdots+x_r)}}{e^{t_{\sigma(j)}+t_{\sigma(j)+1}+\cdots+t_r}+e^{x_j+\cdots+x_r}-1} dt_jdx_j.
		\end{align*}
		Here, we define
		\begin{align*}
			H(\mathbf{u}_r;\mathbf{s}_r)&=\int_{(\mathcal{C}_\varepsilon)^{2r}}\prod_{j=1}^{r}t_j^{s_j-1}x_j^{u_j-1}\frac{e^{(1-a_{\sigma(j)})(t_{\sigma(j)}+t_{\sigma(j)+1}+\cdots+t_r)}e^{(1-b_j)(x_j+\cdots+x_r)}}{e^{t_{\sigma(j)}+t_{\sigma(j)+1}+\cdots+t_r}+e^{x_j+\cdots+x_r}-1} dt_jdx_j.
		\end{align*}
		By Lemma \ref{fivelemma}, we have
		\begin{align*}
			&\left|\int_{(\mathcal{C}_\varepsilon)^{2r}}\prod_{j=1}^{r}t_j^{s_j-1}x_j^{u_j-1}\frac{e^{(1-a_{\sigma(j)})(t_{\sigma(j)}+t_{\sigma(j)+1}+\cdots+t_r)}e^{(1-b_j)(x_j+\cdots+x_r)}}{e^{t_{\sigma(j)}+t_{\sigma(j)+1}+\cdots+t_r}+e^{x_j+\cdots+x_r}-1} dt_jdx_j\right|\\
			\leq&\sum_{\substack{I_1,\ldots,I_{2r}\\=C(0;\varepsilon)~\mm{or}~(\varepsilon,\infty)}}{}\prod_{\substack{l_1=1,\ldots,r\\ I_{l_1}=(\varepsilon,\infty)}}^{}\prod_{\substack{l_2=r+1,\ldots,2r\\ I_{l_2}=(\varepsilon,\infty)}}^{}(e^{2\pi iu_{l_1}}-1)(e^{2\pi is_{l_2}}-1)\prod_{\substack{j_1\in\{1,\ldots,r\}\\ I_{j_1}=(\varepsilon,\infty)}}^{}\int_{I_{j_1}}^{}\left|\frac{x_{j_1}^{u_{j_1}-1}}{e^{\frac{1}{4}x_{j_1}}}\right|dx_{j_1}\\
			&\times\prod_{\substack{j_2\in\{1,\ldots,r\}\\ I_{j_2}=C(0;\varepsilon)}}^{}\int_{I_{j_2}}^{}\left|\frac{x_{j_2}^{u_{j_2}-1}}{e^{\frac{1}{4}x_{j_2}}}\right|
			|dx_{j_2}|\prod_{\substack{j_3\in\{r+1,\ldots,2r\}\\ I_{j_3}=(\varepsilon,\infty)}}^{}\int_{I_{j_3}}^{} \left|\frac{t_{j_3}^{s_{j_3}-1}}{e^{\frac{1}{4}t_{j_3}}}
			\right|dt_{j_3}\prod_{\substack{j_4\in\{r+1,\ldots,2r\}\\ I_{j_4}=C(0;\varepsilon)}}^{}\int_{I_{j_4}}^{} \left|\frac{t_{j_4}^{s_{j_4}-1}}{e^{\frac{1}{4}t_{j_4}}}
			\right||dt_{j_4}|.
		\end{align*}
		Hence $H(\mathbf{u}_r;\mathbf{s}_r)$ converges absolutely. Moreover, except for the case $I_1=\cdots=I_{2r}=(\varepsilon,\infty)$, it holds that
		\begin{align*}
			&\left|\int_{I_1\times\cdots\times I_{2r}}\prod_{j=1}^{r}t_j^{s_j-1}x_j^{u_j-1}\frac{e^{(1-a_{\sigma(j)})(t_{\sigma(j)}+t_{\sigma(j)+1}+\cdots+t_r)}e^{(1-b_j)(x_j+\cdots+x_r)}}{e^{t_{\sigma(j)}+t_{\sigma(j)+1}+\cdots+t_r}+e^{x_j+\cdots+x_r}-1} dt_jdx_j\right|\\
			\leq &M_3\prod_{\substack{j=1,\ldots,r\\ I_j=C(0;\varepsilon)}}^{}\prod_{\substack{l=r+1,\ldots,2r\\ I_l=C(0;\varepsilon)}}^{} \varepsilon^{\mm{Re}(u_j+s_l)}\rightarrow0~(\varepsilon\rightarrow0).
		\end{align*}
		Therefore, we have
		\begin{align*}
			&H(\mathbf{u}_r;\mathbf{s}_r)\\
			&=\prod_{j=1}^{r}(e^{2\pi iu_j}-1)(e^{2\pi is_j}-1)\int_{(0,\infty)^{2r}}\prod_{j=1}^{r}t_j^{s_j-1}x_j^{u_j-1}\frac{e^{(1-a_{\sigma(j)})(t_{\sigma(j)}+t_{\sigma(j)+1}+\cdots+t_r)}e^{(1-b_j)(x_j+\cdots+x_r)}}{e^{t_{\sigma(j)}+t_{\sigma(j)+1}+\cdots+t_r}+e^{x_j+\cdots+x_r}-1} dt_jdx_j,
		\end{align*}
		and we get the integral expression
		\begin{align*}
			\eta(\mathbf{u}_r;\mathbf{s}_r;\sigma;\mathbf{a}_r;\mathbf{b}_r)=\prod_{j=1}^{r}\frac{1}{(e^{2\pi iu_j}-1)(e^{2\pi is_j}-1)\Gamma(u_j)\Gamma(s_j)}H(\mathbf{u}_r;\mathbf{s}_r).
		\end{align*}
		This gives the analytic continuation of $\eta(\mathbf{u}_r;\mathbf{s}_r;\sigma;\mathbf{a}_r;\mathbf{b}_r)$ to the region $\Co^{2r}$.
		
	\end{proof}
	
	The duality formula is as follows.
	
	\begin{theorem}\label{firstdual}
		For $\mathbf{u}_r, \mathbf{s}_r\in\Co^r, \sigma\in\mathfrak{S}_r, \mathbf{a}_r\in\Z_{\geq0}^r$ with $a_1, a_{\sigma^{-1}(1)}\geq1$ and $\mathbf{b}_r\in\Z_{\geq0}^r$ with $b_1, b_{\sigma^{-1}(1)}\geq1$ and $a_{\sigma(j)}+b_j\geq0 $, we have
		\[
		\eta(\mathbf{u}_r;\mathbf{s}_r;\sigma;\mathbf{a}_r;\mathbf{b}_r)=\eta(\mathbf{s}_r;\mathbf{u}_r;\sigma^{-1};\mathbf{b}_r;\mathbf{a}_r).
		\]
	\end{theorem}
	
	\begin{proof}
		By the integral expression \eqref{contourinteta}, we have
		\begin{align*}
			&\eta(\mathbf{u}_r;\mathbf{s}_r;\sigma;\mathbf{a}_r;\mathbf{b}_r)\\
			=&\prod_{j=1}^{r}\frac{1}{(e^{2\pi iu_j}-1)(e^{2\pi is_j}-1)\Gamma(u_j)\Gamma(s_j)}\int_{\mathcal{C}^{2r}}^{}\prod_{j=1}^{r}t_j^{s_j-1}x_j^{u_j-1}
			\frac{e^{(1-a_{\sigma(j)})(t_{\sigma(j)}+t_{\sigma(j)+1}+\cdots+t_r)}e^{(1-b_j)(x_j+\cdots+x_r)}}{e^{t_{\sigma(j)}+t_{\sigma(j)+1}+\cdots+t_r}+e^{x_j+\cdots+x_r}-1} 
			dt_jdx_j\\
			=&\prod_{j=1}^{r}\frac{1}{(e^{2\pi is_j}-1)(e^{2\pi iu_j}-1)\Gamma(s_j)\Gamma(u_j)}\int_{\mathcal{C}^{2r}}^{}\prod_{j=1}^{r}x_j^{u_j-1}t_j^{s_j-1}
			\frac{e^{(1-b_{\sigma^{-1}(j)})(x_{\sigma^{-1}(j)}+x_{\sigma^{-1}(j)+1}+\cdots+x_r)}e^{(1-a_{j})(t_{j}+\cdots+t_r)}}{e^{x_{\sigma^{-1}(j)}+x_{\sigma^{-1}(j)+1}\cdots+x_r}+e^{t_{j}+\cdots+t_r}-1} 
			dx_jdt_j\\
			=&\eta(\mathbf{s}_r;\mathbf{u}_r;\sigma^{-1};\mathbf{b}_r;\mathbf{a}_r).
		\end{align*}
	Hence we obtain Theorem \ref{firstdual}.
	\end{proof}
	
	\begin{corollary}[{\cite[Proposition 2.5]{Ya}}]
		For $\mathbf{u}_r, \mathbf{s}_r\in\Co^r$, we have
		\[
		\eta(\mathbf{u}_r;\mathbf{s}_r)=\eta(\mathbf{s}_r;\mathbf{u}_r).
		\]
	\end{corollary}
	
	\begin{proof}
		
		Since 
		\[
		{\mm{Li}}_{\mathbf{u}_r}^{\cy}(\mathbf{z}_r;1,\ldots,1)={\mm{Li}}_{\mathbf{u}_r}^{\cy}(\mathbf{z}_r),
		\]
		the result follows by putting $\sigma=\mm{id}, \mathbf{a}_r=\mathbf{b}_r=(1,\ldots,1)$ in Theorem \ref{firstdual}.
	\end{proof}
	
	\section{Values of $\eta(\mathbf{u}_r;\mathbf{s}_r;\sigma;\mathbf{a}_r;\mathbf{b}_r)$ for some particular case}
	
	In this section, we consider the values of $\eta(\mathbf{u}_r;\mathbf{s}_r;\sigma;\mathbf{a}_r;\mathbf{b}_r)$ for some particular case. For values at non-positive integers, we have the following proposition.
	
	\begin{proposition}\label{etaber}
		For $\mathbf{k}_r\in\Z^r, -\mathbf{n}_r=(-n_1,\ldots,-n_r)\in\Z_{\leq0}^r, \sigma\in\mathfrak{S}_r, \mathbf{a}_r\in\Z_{\geq0}^r$ with $a_1, a_{\sigma^{-1}(1)}\geq1$ and $\mathbf{b}_r\in\Z_{\geq0}^r$ with $b_1, b_{\sigma^{-1}(1)}\geq1$ and $a_j+b_j\geq1$, we have
		\begin{align}
			&\eta(\mathbf{k}_r;-\mathbf{n}_r;\sigma;\mathbf{a}_r;\mathbf{b}_r)=B_{\mathbf{n}_r}^{(\mathbf{k}_r)}(\sigma;\mathbf{a}_r;\mathbf{b}_r).\label{multinterpolate1}
		\end{align}
	\end{proposition}
	
	\begin{proof}
		Since the argument for the proof is similar to that of \cite[Theorem 5.7]{KT1}, we omit the proof. 
		
	\end{proof}
	
	\begin{example}
		We consider the case $-\mathbf{k}_r=(-k_1,\ldots,-k_r)\in\Z_{\leq0}^r, -\mathbf{n}_r=(-n_1,\ldots,-n_r)\in\Z_{\leq0}^r, \sigma=\mm{id}, \mathbf{a}_r=(1,0,0,\ldots,0), \mathbf{b}_r=(1,\ldots,1)$. 
		Then, by Theorem \ref{firstdual} and Proposition \ref{etaber}, we have
		\[
		B_{\mathbf{n}_r}^{(-\mathbf{k}_r)}(\mm{id};1,0,0,\ldots,0;1,\ldots,1)=B_{\mathbf{k}_r}^{(-\mathbf{n}_r)}({\mm{id}}^{-1};1,\ldots,1;1,0,0,\ldots,0).
		\]
		For $B_{\mathbf{n}_r}^{(-\mathbf{k}_r)}(\mm{id};1,0,0,\ldots,0;1,\ldots,1)$, the generating function is 
		\[
		\frac{{\mm{Li}}_{-\mathbf{k}_r}^{\textup{\foreignlanguage{russian}{\cyrsh}},\star}(1-e^{-t_1-\cdots-t_r},\ldots,1-e^{-t_r})}{1-e^{-t_1-\cdots-t_r}},
		\]
		which is similar to that of $\mathbb{B}_{\mathbf{m}_r}^{\star,(\mathbf{u}_r)}$ defined by Baba, Nakasuji and Sakata. On the other hand, for 
		$B_{\mathbf{k}_r}^{(-\mathbf{n}_r)}({\mm{id}}^{-1};1,\ldots,1;1,0,0,\ldots,0)$, the generating function is
		\[
		\prod_{j=2}^{r}e^{-(t_j+\cdots+t_r)}\frac{{\mm{Li}}_{-\mathbf{n}_r}^{\textup{\foreignlanguage{russian}{\cyrsh}}}(1-e^{-t_1-\cdots-t_r},\ldots,1-e^{-t_r})}{1-e^{-t_1-\cdots-t_r}}.
		\]
		Hence, by \eqref{itocber}, we have
		\[
		B_{\mathbf{k}_r}^{(-\mathbf{n}_r)}({\mm{id}}^{-1};1,\ldots,1;1,0,0,\ldots,0)=\sum_{m_1=0}^{k_1}\cdots\sum_{m_r=0}^{k_r}\binom{k_1}{m_1}\cdots\binom{k_r}{m_r}C_{k_1-m_1,\ldots,k_r-m_r}^{(-n_1,\ldots,-n_r),(r)},
		\]
		and
		\[
		B_{\mathbf{n}_r}^{(-\mathbf{k}_r)}(\mm{id};1,0,0,\ldots,0;1,\ldots,1)=\sum_{m_1=0}^{k_1}\cdots\sum_{m_r=0}^{k_r}\binom{k_1}{m_1}\cdots\binom{k_r}{m_r}C_{k_1-m_1,\ldots,k_r-m_r}^{(-n_1,\ldots,-n_r),(r)}.
		\]
	\end{example}
	
	For the values at positive integers, we have the following proposition. For simplicity, we put
	\begin{align*}
		\eta(\mathbf{u}_r;\mathbf{s}_r;\sigma;\mathbf{a}_r;1,\ldots,1)=\eta^{\star}(\mathbf{u}_r;\mathbf{s}_r;\sigma;\mathbf{a}_r)
	\end{align*}
	for $(a_1,\ldots,a_r)\in\{0,1\}^r$ with $a_1=1$, and
	\begin{align*}
		\eta(\mathbf{u}_r;\mathbf{s}_r;\sigma;1,\ldots,1;\mathbf{b}_r)=\eta^{\star\star}(\mathbf{u}_r;\mathbf{s}_r;\sigma;b_1,\ldots,b_r)
	\end{align*}
	for $ (b_1,\ldots,b_r)\in\{0,1\}^r$ with $b_1=1$. Note that, with this notation, Theorem \ref{firstdual} can be written as
	\begin{align}\label{eplidual}
		\eta^{\star}(\mathbf{k}_r;\mathbf{n}_r;\sigma;a_1,\ldots,a_r)=\eta^{\star\star}(\mathbf{n}_r;\mathbf{k}_r;\sigma^{-1};a_1,\ldots,a_r)
	\end{align}
	with $a_1,\ldots,a_r\in\Z_{\geq0}$ with $(a_1,a_{\sigma^{-1}(1)})=(1,1)$.
	\begin{theorem}\label{zetaeta}
		For $\mathbf{k}_r, \mathbf{n}_r\in\Z_{\geq1}^r, \sigma\in\mathfrak{S}_r, \mathbf{a}_r=\in\left\{0,1\right\}^r$ with
		$(a_1, a_{\sigma^{-1}(1)})=(1,1)$, we have
		\[\eta^{\star}(\mathbf{k}_r;\mathbf{n}_r;\sigma;\mathbf{a}_r), \eta^{\star\star}(\mathbf{k}_r;\mathbf{n}_r;\sigma;\mathbf{a}_r)\in\mathcal{Z}_{k_1+\cdots+k_r+n_1+\cdots+n_r}.\]
	\end{theorem}
	
	\begin{proof}
		We use the notation in \cite{It1}. We write
		\[
		{\mm{Li}}_{\mathbf{k}_r}^{\cy}\left(\mathbf{z}_r;\mathbf{a}_r\right)=\sum_{p=1}^{r}\sum_{1= l_1<\cdots<l_p\leq r}^{} c(l(k_1,\ldots,k_r);\mathbf{a}_r){\mm{Li}}_{l(k_1,\ldots,k_r)}^{\cy}\left(z_{l_1},\ldots,z_{l_p}\right),
		\]
		where 
		\[
		l(k_1,\ldots,k_r)=(k_1+\cdots+k_{l_1}, k_{l_1+1}+\cdots+k_{l_2}, \ldots, k_{l_{p}+1}+\cdots+k_{r})
		\]
		and $c(l(k_1,\ldots,k_r);\mathbf{a}_r)\in\Z$. For the last entry, if $p=r$, $k_{l_{p}+1}+\cdots+k_{r}=k_r$. We consider the value $\eta^{\star}(\mathbf{k}_r;\mathbf{n}_r;\sigma;\mathbf{a}_r)$. By \eqref{intfiveeta}, \eqref{polyhyp} and changing variables $x_j=1-e^{-t_j-\cdots-t_r}$, we have 
		\begin{align*}
			&\eta^{\star}(\mathbf{k}_r;\mathbf{n}_r;\sigma;\mathbf{a})\\
			=&\prod_{j=1}^{r}\frac{1}{\Gamma(n_j)}\int_{(0,\infty)^r}^{}\prod_{j=1}^{r}t_j^{n_j-1}\frac{{\mm{Li}}_{\mathbf{k}_r}^{\cy}(1-e^{t_{\sigma(1)}+t_{\sigma(1)+1}+\cdots+t_r},\ldots,1-e^{t_{\sigma(r)}+t_{\sigma(r)+1}+\cdots+t_r};\mathbf{a}_r)}{(1-e^{t_{\sigma(1)}+t_{\sigma(1)+1}+\cdots+t_r})^{a_1}\cdots(1-e^{t_{\sigma(r)}+t_{\sigma(r)+1}+\cdots+t_r})^{a_r}}dt_1\cdots dt_r\\
			=&\sum_{p=1}^{r}\sum_{1\leq l_1<\cdots<l_p\leq r}^{} c(l(k_1,\ldots,k_r);\mathbf{a}_r)\prod_{j=1}^{r}\frac{1}{\Gamma(n_j)}\int_{0<x_r<\cdots<x_1<1}^{}\prod_{j=1}^{r}({\mm{Li}}_{1}(x_j)-{\mm{Li}}_{1}(x_{j+1}))^{n_j-1}\\
			&\five\times {\mm{Li}}_{l(k_1,\ldots,k_r)}^{\cy}\left(\frac{1}{1-x_{\sigma(l_1)}^{-1}},\ldots,\frac{1}{1-x_{\sigma(l_p)}^{-1}}\right)\prod_{j=1}^{r}\left(-\frac{1}{x_{\sigma(j)}}\right)^{a_j}\left(\frac{1}{1-x_{\sigma(j)}}\right)^{1-a_j}dx_j\\
			=&\sum_{p=1}^{r}\sum_{1\leq l_1<\cdots<l_p\leq r}^{} c(l(k_1,\ldots,k_r);\mathbf{a}_r)\prod_{j=1}^{r}\frac{1}{\Gamma(n_j)}\int_{0<x_r<\cdots<x_1<1}^{}(I(0;1;x_{j+1})-I(0;1;x_{j}))^{n_j-1}\\
			&\times I(0;1-x_{\sigma(l_1)}^{-1},\{0\}_{k_1+\cdots+k_{l_1}-1},\ldots,
			1-x_{\sigma(l_p)}^{-1},\{0\}_{k_{l_{p}+1}+k_{r}-1};1)\prod_{j=1}^{r}\left(-\frac{1}{x_{\sigma(j)}}\right)^{a_j}\left(\frac{1}{1-x_{\sigma(j)}}\right)^{1-a_j}dx_j\\
			=&\sum_{p=1}^{r}\sum_{1\leq l_1<\cdots<l_p\leq r}^{} c(l(k_1,\ldots,k_r);\mathbf{a}_r)\prod_{j=1}^{r}\frac{1}{\Gamma(n_j)}\int_{0<x_r<\cdots<x_1<1}^{}(I(0;1;x_{j+1})-I(0;1;x_{j}))^{n_j-1}\\
			&\times I(0;\{1\}_{k_{l_{p}+1}+k_{r}-1}, x_{\sigma(l_p)}^{-1},\ldots,
			\{1\}_{k_1+\cdots+k_{l_1}-1},x_{\sigma(l_1)}^{-1};1)\prod_{j=1}^{r}\left(-\frac{1}{x_{\sigma(j)}}\right)^{a_j}\left(\frac{1}{1-x_{\sigma(j)}}\right)^{1-a_j}dx_j.
		\end{align*}
		Hence, by Lemma \ref{itlem2} and 
		\[
		\int_{0}^{x_{j-1}}\frac{1}{1-x_j}dx_j=-I(0;1;x_{j-1}),
		\]
		we obtain $\eta^{\star}(\mathbf{k}_r;\mathbf{n}_r;\sigma;\mathbf{a}_r)\in\mathcal{Z}_{k_1+\cdots+k_r+n_1+\cdots+n_r}$. Also, by the similar calculation, we obtain $\eta^{\star\star}(\mathbf{k}_r;\mathbf{n}_r;\sigma;\mathbf{a}_r)\in\mathcal{Z}_{k_1+\cdots+k_r+n_1+\cdots+n_r}$.
	\end{proof}
	
	\begin{example}
		
		For the case $r=2, n_1=n_2=k_1=k_2=1, \sigma=\mm{id}, a_1=1,a_2=0$ in Theorem \ref{firstdual} (or \eqref{eplidual}), we have
		\[
		\eta^{\star}(1,1;1,1;\mm{id};1,0)=\eta^{\star\star}(1,1;1,1;\sigma^{-1};1,0).
		\]
		Now we calculate these values explicitly by variable removing. By changing variable $1-e^{-t_j-\cdots-t_r}=x_j$, we have
		\begin{align*}
			\eta^{\star}(1,1;1,1;\mm{id};1,0)&=\int_{(0,\infty)^2}^{} \frac{{\mm{Li}}_{1,1}^{\cy,\star}(1-e^{t_1+t_2},1-e^{t_2})}{1-e^{t_1+t_2}}dt_1dt_2\\
			&=\int_{(0,\infty)^2}^{} \frac{{\mm{Li}}_{1,1}^{\cy}(1-e^{t_1+t_2},1-e^{t_2})}{1-e^{t_1+t_2}}dt_1dt_2+\int_{(0,\infty)^2}^{} \frac{{\mm{Li}}_{2}(1-e^{t_1+t_2})}{1-e^{t_1+t_2}}dt_1dt_2\\
			&=\int_{0<x_2<x_1<1}^{} \frac{{\mm{Li}}_{1,1}^{\cy}\left(\frac{1}{1-x_1^{-1}},\frac{1}{1-x_2^{-1}}\right)}{x_1(x_2-1)}dx_1dx_2+\int_{0<x_2<x_1<1}^{} \frac{{\mm{Li}}_{2}
				\left(\frac{1}{1-x_1^{-1}}\right)}{x_1(x_2-1)}dx_1dx_2\\
			&=\int_{0<x_2<x_1<1}^{} \frac{I(0;x_2^{-1},x_1^{-1};1)}{x_1(x_2-1)}dx_1dx_2-\int_{0<x_2<x_1<1}^{} \frac{I(0;x_1,1;x_1)}{x_1(x_2-1)}dx_1dx_2.
		\end{align*}
		For ${\mm{Li}}_{1,1}^{\cy}(x_2,x_1)=I(0;x_2^{-1},x_1^{-1};1)$, by variable removing, we obtain
		\begin{align*}
			I(0;x_2^{-1},x_1^{-1};1)&=I(0;x_1;x_2)I(0;1;x_1)+I(0;1,0;x_2)-I(0;1,x_1;x_2),\\
			I(0;1,x_1,1;x_1)&=I(0;1,1,1;x_1)-I(0;1,0,1;x_1).
		\end{align*}
		Hence we have
		\begin{align*}
			\eta^{\star}(1,1;1,1;\mm{id};1,0)&=\int_{0<x_2<x_1<1}^{} \frac{I(0;x_1;x_2)I(0;1;x_1)+I(0;1,0;x_2)-I(0;1,x_1;x_2)}{x_1(x_2-1)}dx_1dx_2\\
			&-\int_{0<x_2<x_1<1}^{} \frac{I(0;x_1,1;x_1)}{x_1(x_2-1)}dx_1dx_2\\
			&=\int_{0}^{1} \frac{I(0;x_1,1;x_1)I(0;1;x_1)+I(0;1,0,1;x_1)-I(0;1,x_1,1;x_1)}{x_1}dx_1\\
			&-\int_{0}^{1} \frac{I(0;x_1,1;x_1)I(0;1;x_1)}{x_1}dx_1\\
			&=\int_{0}^{1} \frac{I(0;1,0,1;x_1)-I(0;1,x_1,1;x_1)}{x_1}dx_1\\
			&=\int_{0}^{1} \frac{2I(0;1,0,1;x_1)-I(0;1,1,1;x_1)}{x_1}dx_1\\
			&=\zeta(1,1,2)+2\zeta(2,2).
		\end{align*}
		For $\eta^{\star\star}(1,1;1,1;\mm{id}^{-1};1,0)$, a similar calculation shows
		\begin{align*}
			I(0;x_1,1,0;x_1)&=I(0;1,0,1;x_1)-2I(0;1,0,0;x_1)+I(0;1,1,0;x_1)+I(0;1;x_1)\zeta(2),\\
			I(0;x_1,0,1;x_1)&=-I(0,1,1,0;x_1)+I(0;1,0,0;x_1)-\zeta(2)I(0;1;x_1).
		\end{align*}
		Hence we have
		\begin{align*}
			&\eta^{\star\star}(1,1;1,1;{\mm{id}}^{-1};1,0)\\
			=&-\int_{(0,\infty)^2}^{} \frac{{\mm{Li}}_{1,1}^{\cy}(1-e^{t_1+t_2},1-e^{t_2})}{(1-e^{t_1+t_2})(1-e^{-t_2})}dt_1dt_2\\
			=&\int_{0<x_2<x_1<1}^{} \frac{{\mm{Li}}_{1,1}^{\cy}\left(\frac{1}{1-x_1^{-1}},\frac{1}{1-x_2^{-1}}\right)}{x_1x_2}dx_1dx_2+\int_{0<x_2<x_1<1}^{} \frac{{\mm{Li}}_{1,1}^{\cy}
				\left(\frac{1}{1-x_1^{-1}},\frac{1}{1-x_2^{-1}}\right)}{x_1(1-x_2)}dx_1dx_2\\
			=&\int_{0<x_2<x_1<1}^{} \frac{I(0;x_2^{-1},x_1^{-1};1)}{x_1x_2}dx_1dx_2+\int_{0<x_2<x_1<1}^{} \frac{I(0;x_2^{-1},x_1^{-1};1)}{x_1(1-x_2)}dx_1dx_2\\\\
			=&\int_{0<x_2<x_1<1}^{} \frac{I(0;x_1;x_2)I(0;1;x_1)+I(0;1,0;x_2)-I(0;1,x_1;x_2)}{x_1x_2}dx_1dx_2\\
			&\five-\int_{0<x_2<x_1<1}^{} \frac{I(0;x_1;x_2)I(0;1;x_1)+I(0;1,0;x_2)-I(0;1,x_1;x_2)}{x_1(x_2-1)}dx_1dx_2\\
			=&\int_{0}^{1} \frac{I(0;x_1,0;x_1)I(0;1;x_1)+I(0;1,0,0;x_1)-I(0;1,x_1,0;x_1)}{x_1}dx_1\\
			&\five-\int_{0}^{1} \frac{I(0;x_1,1;x_1)I(0;1;x_1)+I(0;1,0,1;x_1)-I(0;1,x_1,1;x_1)}{x_1}dx_1\\
			=&\int_{0}^{1} \frac{-2I(0;1,1,1;x_1)+2I(0;1,1,0;x_1)}{x_1}dx_1\\
			=&2\zeta(1,1,2)+2\zeta(1,3).
		\end{align*}
		Therefore, we obtain
		\[
		\zeta(1,1,2)+2\zeta(2,2)=2\zeta(1,3)+2\zeta(1,1,2).
		\]
	\end{example}

	\begin{example}\label{examplewt5}
		
		For $r=2, k_1=k_2=1, n_1=1, n_2=2, \sigma=(1,2)\in\mathfrak{S}_2$ and $a_1=a_2=1$, we have
		\begin{align*}
			I(0;x_1^{-1},x_2^{-1};1)&=I(0;1,x_1;x_2)-I(0;1,0;x_2)+I(0;1;x_2)I(0;1;x_1)-I(0;x_1;x_2)I(0;1;x_1),\\
			I(0;1,x_1,0,1;x_1)&=-I(1,1,0,1;x_1)+2I(1,0,0,1;x_1)-2I(0;1,1;x_1)\zeta(2)\\
			&\five+2I(1,0,1,0;x_1)-2I(1,0,1,1;x_1)-I(1,1,1,0;x_1),\\
			I(0;1,x_1,1,0;x_1)&=I(1,1,0,1;x_1)-3I(1,0,0,1;x_1)+2I(0;1,1;x_1)\zeta(2)\\
			&\five-I(1,0,1,0;x_1)+2I(1,0,1,1;x_1)+I(1,1,1,0;x_1),\\
			I(0;x_1,1,1,0;x_1)&=-2I(0;1,1,0,0;x_1)+2I(1,0,0,1;x_1)-I(0;1,1;x_1)\zeta(2)\\
			&\five-I(1,0,1,1;x_1)+I(1,1,1,0;x_1),\\
			I(0;x_1,1,0,1;x_1)&=+2I(1,1,0,1;x_1)-3I(1,0,0,1;x_1)+2I(1,1;x_1)\zeta(2)\\
			&\five-I(1,0,1,0;x_1)+2I(1,0,1,1;x_1)-I(1,1,1,0;x_1)
		\end{align*}
		and
		\begin{align*}
			&\eta^{\star}(1,1;1,2;\sigma;1,1)\\
			&=\int_{(0,\infty)^2}^{}t_2\frac{{\mm{Li}}_{1,1}^{\cy}(1-e^{t_2},1-e^{t_1+t_2})}{(1-e^{t_1+t_2})(1-e^{t_2})}dt_1dt_2\\
			&=-\int_{0<x_2<x_1<1}^{}\frac{I(0;1;x_2)I(0;x_1^{-1},x_2^{-1};1)}{x_1x_2}dx_1dx_2\\
			&=-2\int_{0<x_2<x_1<1}^{}\frac{I(0;1,1,x_1;x_2)}{x_1x_2}dx_1dx_2-\int_{0<x_2<x_1<1}^{}\frac{I(0;1,x_1,1;x_2)}{x_1x_2}dx_1dx_2\\
			&\five+2\int_{0<x_2<x_1<1}^{}\frac{I(0;1,1,0;x_2)}{x_1x_2}dx_1dx_2+\int_{0<x_2<x_1<1}^{}\frac{I(0;1,0,1;x_2)}{x_1x_2}dx_1dx_2\\
			&\five-2\int_{0<x_2<x_1<1}^{}\frac{I(0;1,1;x_2)I(0;1;x_1)}{x_1x_2}dx_1dx_2+\int_{0<x_2<x_1<1}^{}\frac{I(0;1,x_1;x_2)I(0;1;x_1)}{x_1x_2}dx_1dx_2\\
			&\five+\int_{0<x_2<x_1<1}^{}\frac{I(0;x_1,1;x_2)I(0;1;x_1)}{x_1x_2}dx_1dx_2\\
			&=-2\int_{0}^{1}\frac{I(0;1,1,x_1,0;x_1)}{x_1}dx_1-\int_{0}^{1}\frac{I(0;1,x_1,1,0;x_1)}{x_1}dx_1+\int_{0}^{1}
			\frac{I(0;x_1,1,0;x_1)I(0;1;x_1)}{x_1}dx_1\\
			&\five-2\int_{0}^{1}\frac{I(0;1,1,0;x_1)I(0;1;x_1)}{x_1}dx_1+\int_{0}^{1}\frac{I(0;1,x_1,0;x_1)I(0;1;x_1)}{x_1}dx_1+2\zeta(1,4)+\zeta(2,3)\\
			&=\int_{0}^{1}\frac{I(0;1,x_1,0,1;x_1)}{x_1}dx_1+\int_{0}^{1}\frac{I(0;1,x_1,1,0;x_1)}{x_1}dx_1+2\int_{0}^{1}\frac{I(0;x_1,1,1,0;x_1)}{x_1}dx_1\\
			&\five+\int_{0}^{1}\frac{I(0;x_1,1,0,1;x_1)}{x_1}dx_1+2\zeta(1,4)+\zeta(2,3)+6\zeta(1,1,3)+2\zeta(1,2,2)\\
			&=-2\zeta(1,4)+\zeta(2,3)+4\zeta(1,1,3).
		\end{align*}
		On the other hand, we have
		\begin{align*}
			I(0;1,x_1^{-1},x_2^{-1};1)&=I(0;1,1,x_1;x_2)-I(0;1,0,x_1;x_2)-I(0;1,1,0;x_2)-I(0;1,0,0;x_2)\\
			&\five+I(0;1;x_2)I(0;1,1;x_1)-I(0;1;x_2)I(0;1,0;x_1)\\
			&\five-I(0;x_1;x_2)I(0;1,1;x_1)+I(0;x_1;x_2)I(0;1,0;x_1),\\
			I(0;x_1,0,1,1;x_1)&=-I(0;1,1,1,0;x_1)+I(0;1,1,0,0;x_1)-I(0;1,1,0,1;x_1)\\
			&\five+I(0;1,0,0,1;x_1)-I(0;1,1;x_1)\zeta(2),\\
			I(0;x_1,0,1,0;x_1)&=-I(0;1,0,1,0;x_1)+3I(0;1,0,0,0;x_1)-2I(0;1,1,0,0;x_1)-2I(0;1,0;x_1)\zeta(2),\\
			I(0;x_1,1,0,0;x_1)&=I(0;1,0,1,0;x_1)-3I(0;1,0,0,0;x_1)+I(0;1,1,0,0;x_1)\\
			&\five+I(0;1,0,0,1;x_1)+I(0;1;x_1)\zeta(3)+I(0;1,0;x_1)\zeta(2),\\
			I(0;1,x_1,0,0;x_1)&=-I(0;1,0,0,1;x_1)+I(0;1,0,0,0;x_1)-I(0;1;x_1)\zeta(3)
		\end{align*}
		and
		\begin{align*}
			&\eta^{\star\star}(1,2;1,1;\sigma^{-1};1,1)\\
			&=\int_{(0,\infty)^2}^{}\frac{{\mm{Li}}_{1,2}^{\cy}(1-e^{t_2},1-e^{t_1+t_2})}{(1-e^{t_1+t_2})(1-e^{t_2})}dt_1dt_2\\
			&=-\int_{0<x_2<x_1<1}^{}\frac{I(0;1,x_1^{-1},x_2^{-1};1)}{x_1x_2}dx_1dx_2\\
			&=\zeta(1,2,2)+\zeta(2,3)+3\zeta(1,1,3)+4\zeta(1,4)-\zeta(3,2).
		\end{align*}
		Hence we obtain
		\begin{align*}
			&-2\zeta(1,4)+\zeta(2,3)+4\zeta(1,1,3)=\zeta(1,2,2)+\zeta(2,3)+3\zeta(1,1,3)+4\zeta(1,4)-\zeta(3,2)\\
			&\rightarrow 6\zeta(1,4)-\zeta(3,2)-\zeta(1,1,3)+\zeta(1,2,2)=0.
		\end{align*}
	\end{example}
	
	\begin{remark}
		In the same way, for the identity permutation $\mm{id}$ and $\sigma=(1,2)$, we have
		\begin{align*}
			&\eta^{\star}(1;4;\mm{id};1)=\eta^{\star\star}(4;1;\mm{id};1)~(\Leftrightarrow \eta(1;4)=\eta(4;1))\\
			&\rightarrow \zeta(5)+\zeta(1,4)+\zeta(2,3)+\zeta(3,2)+\zeta(1,1,3)+\zeta(1,2,2)+\zeta(2,1,2)-3\zeta(1,1,1,2)=0,\\
			&\eta^{\star}(2;3;\mm{id};1)=\eta^{\star\star}(3;2;\mm{id};1)~(\Leftrightarrow \eta(2;3)=\eta(3;2))\\
			&\rightarrow 2\zeta(1,4)+\zeta(2,3)+\zeta(3,2)+\zeta(1,2,2)+\zeta(2,1,2)-2\zeta(1,1,1,2)=0,\\
			&\eta^{\star}(1,1;1,2;\mm{id};1,1)=\eta^{\star\star}(1,2;1,1;\mm{id};1,1)~(\Leftrightarrow \eta(1,1;1,2)=\eta(1,2;1,1))\\
			&\rightarrow 2\zeta(1,4)-\zeta(3,2)+3\zeta(1,1,3)+\zeta(1,2,2)=0,\\
			&\eta^{\star}(1,1;2,1;\mm{id};1,1)=\eta^{\star\star}(2,1;1,1;\mm{id};1,1)~(\Leftrightarrow \eta(1,1;2,1)=\eta(2,1;1,1))\\
			&\rightarrow -2\zeta(1,4)-\zeta(2,3)-3\zeta(1,1,3)+\zeta(2,1,2)=0,\\
			&\eta^{\star}(1,1;2,1;\sigma;1,1)=\eta^{\star\star}(2,1;1,1;\sigma^{-1};1,1)\\
			&\rightarrow 4\zeta(1,4)-2\zeta(2,3)-\zeta(3,2)+\zeta(1,1,3)+3\zeta(1,2,2)=0.
		\end{align*}
		Combining these with Example \ref{examplewt5}, we obtain
		\begin{align*}
			\zeta(5)&=\zeta(1,1,1,2),\\
			4\zeta(1,4)&=2\zeta(2,1,2)-\zeta(1,1,1,2),\\
			4\zeta(2,3)&=-6\zeta(2,1,2)+5\zeta(1,1,1,2),\\
			\zeta(3,2)&=\zeta(2,1,2),\\
			4\zeta(1,1,3)&=2\zeta(2,1,2)-\zeta(1,1,1,2),\\
			4\zeta(1,2,2)&=-6\zeta(2,1,2)+5\zeta(1,1,1,2).
		\end{align*}
	\end{remark}
	
	\begin{example}
		
		For $\mm{id}$ and $\sigma=(1,2)$, by the relations
		\begin{enumerate}
			\item $\eta^{\star}(1;5;\mm{id};1)=\eta^{\star\star}(5;1;\mm{id}^{-1};1)~(\Leftrightarrow \eta(1;5)=\eta(5;1) ),$
			\item $\eta^{\star}(2;4;\mm{id};1)=\eta^{\star\star}(4;2;\mm{id}^{-1};1)~(\Leftrightarrow \eta(2;4)=\eta(4;2) ),$
			\item $\eta^{\star}(1,1;2,2;\mm{id};1,1)=\eta^{\star\star}(2,2;1,1;\mm{id}^{-1};1,1)~(\Leftrightarrow \eta(1,1;2,2)=\eta(2,2;1,1) ),$
			\item $\eta^{\star}(1,2;2,1;\mm{id};1,1)=\eta^{\star\star}(2,1;1,2;\mm{id}^{-1};1,1)~(\Leftrightarrow \eta(1,2;2,1)=\eta(2,1;1,2) ),$
			\item $\eta^{\star}(1,1;1,3;\mm{id};1,1)=\eta^{\star\star}(1,3;1,1;\mm{id}^{-1};1,1)~(\Leftrightarrow \eta(1,1;1,3)=\eta(1,3;1,1) ),$
			\item $\eta^{\star}(1,1;3,1;\mm{id};1,1)=\eta^{\star\star}(3,1;1,1;\mm{id}^{-1};1,1)~(\Leftrightarrow \eta(1,1;3,1)=\eta(3,1;1,1) ),$
			\item $\eta^{\star}(1,1;2,2;\sigma;1,1)=\eta^{\star\star}(2,2;1,1;\sigma^{-1};1,1),$
			\item $\eta^{\star}(1,2;2,1;\sigma;1,1)=\eta^{\star\star}(2,1;1,2;\sigma^{-1};1,1),$
			\item $\eta^{\star}(1,1;1,3;\sigma;1,1)=\eta^{\star\star}(1,3;1,1;\sigma^{-1};1,1),$
			\item $\eta^{\star}(1,1;3,1;\sigma;1,1)=\eta^{\star\star}(3,1;1,1;\sigma^{-1};1,1),$
			\item $\eta^{\star}(1,1,1;1,1,1;\mm{id};1,0,1)=\eta^{\star\star}(1,1,1;1,1,1;\mm{id}^{-1};1,0,1),$
			\item $\eta^{\star}(1,1,1;1,1,1;\mm{id};1,1,0)=\eta^{\star\star}(1,1,1;1,1,1;\mm{id}^{-1};1,1,0),$
			\item $\eta^{\star}(1,2;1,2;\mm{id};1,0)=\eta^{\star\star}(1,2;1,2;\mm{id}^{-1};1,0),$
			\item $\eta^{\star}(2,1;2,1;\mm{id};1,0)=\eta^{\star\star}(2,1;2,1;\mm{id}^{-1};1,0),$
		\end{enumerate}
		we obtain
		
		$\begin{pmatrix}
			1 & 1 & 1 & 1 & 1 & 1 & 1 & 1 & 1 & 1 & 1 & 1 & 1 & 1 & 1 & -4 \\
			0 & 2 & 1 & 1 & 1 & 3 & 3 & 3 & 2 & 2 & 2 & 0 & 1 & 2 & 2 & -5 \\
			0 & 0 & 0 & 1 & 0 & 0 & 1 & -2 & 1 & 0 & 1 & 3 & -3 & -3 & 0 & 0 \\
			0 & 6 & -1 & -1 & 0 & 7 & 7 & 2 & 2 & 0 & -2 & 9 & 3 & 0 & 0 & 0 \\
			0 & 0 & 1 & 1 & -1 & 4 & 1 & -1 & 1 & -1 & 0 & 7 & 2 & 0 & 0 & 0 \\
			0 & 0 & 2 & 1 & 0 & 4 & 1 & 2 & 1 & 1 & 0 & 4 & 3 & 0 & -2 & 0 \\
			0 & -6 & 1 & 2 & 0 & -1 & -2 & -4 & 1 & -1 & 2 & -1 & -6 & -2 & 0 & 0 \\
			0 & 4 & 0 & -1 & 0 & 11 & 3 & 2 & -2 & 0 & -2 & 1 & 6 & 2 & 0 & 0 \\
			0 & 6 & 2 & 0 & -1 & 9 & 4 & 3 & 1 & 0 & -1 & -3 & 2 & 1 & 0 & 0 \\
			0 & 2 & -1 & 0 & 1 & -4 & 2 & 1 & 2 & 1 & -1 & -2 & -2 & -1 & 0 & 0 \\
			0 & -2 & 1 & 0 & 0 & -2 & 0 & -2 & 2 & -2 & 0 & 0 & -3 & 0 & 1 & 0 \\
			0 & 4 & -4 & 1 & 0 & 3 & 3 & 2 & -2 & -2 & 0 & 3 & 1 & 0 & 0 & 0 \\
			0 & 4 & 0 & -1 & 0 & 8 & 3 & 4 & -1 & -2 & -3 & 7 & 7 & 1 & -5 & 4 \\
			0 & -2 & 1 & 0 & 0 & 1 & -2 & 0 & -1 & 0 & -1 & 0 & 3 & -1 & -2 & 2
		\end{pmatrix}
		\begin{pmatrix}
			\zeta(6) \\
			\zeta(1,5) \\
			\zeta(2,4) \\
			\zeta(3,3) \\
			\zeta(4,2) \\
			\zeta(1,1,4) \\
			\zeta(1,2,3) \\
			\zeta(1,3,2) \\
			\zeta(2,1,3) \\
			\zeta(2,2,2) \\
			\zeta(3,1,2) \\
			\zeta(1,1,1,3) \\
			\zeta(1,1,2,2) \\
			\zeta(1,2,1,2) \\
			\zeta(2,1,1,2) \\
			\zeta(1,1,1,1,2) 
		\end{pmatrix}\\
		=
		\begin{matrix}
			\boldsymbol{0}
		\end{matrix}
		$.
	\end{example}
	
	Hence we have
	\begin{align*}
		&\zeta(1,5)-2\zeta(2,4)-\zeta(3,3)+2\zeta(4,2)+2\zeta(1,1,4)-6\zeta(1,2,3)+3\zeta(1,3,2)+3\zeta(2,1,3)-2\zeta(3,1,2)\\
		&+\zeta(1,1,1,3)-2\zeta(1,1,2,2)-\zeta(1,2,1,2)+2\zeta(2,1,1,2)=0,\\
		&48\zeta(6)-14\zeta(1,5)+36\zeta(2,4)+26\zeta(3,3)-31\zeta(1,1,4)+97\zeta(1,2,3)-44\zeta(1,3,2)-44\zeta(2,1,3) \\
		&+9\zeta(2,2,2)+61\zeta(3,1,2)-14\zeta(1,1,1,3)+36\zeta(1,1,2,2)+26\zeta(1,2,1,2)+48\zeta(1,1,1,1,2) =0,
	\end{align*}
	which implies
	
	$
	\begin{pmatrix}
		1 & 0 & 0 & 0 & 0 & 0 & 0 & 0 & 0 & 0 & 0 & 0 & 0 & 0 & 0 & -1 \\
		0 & 24 & 0 & 0 & 0 & 0 & 0 & 0 & 0 & 0 & 0 & 0 & 0 & 0 & -12 & 7 \\
		0 & 0 & -4 & 0 & 0 & 0 & 0 & 0 & 0 & 0 & 0 & 0 & 0 & 0 & -4 & 3 \\
		0 & 0 & 0 & 24 & 0 & 0 & 0 & 0 & 0 & 0 & 0 & 0 & 0 & 0 & 12 & -13 \\
		0 & 0 & 0 & 0 & -1 & 0 & 0 & 0 & 0 & 0 & 0 & 0 & 0 & 0 & 1 & 0 \\
		0 & 0 & 0 & 0 & 0 & 48 & 0 & 0 & 0 & 0 & 0 & 0 & 0 & 0 & -48 & 31 \\
		0 & 0 & 0 & 0 & 0 & 0 & 48 & 0 & 0 & 0 & 0 & 0 & 0 & 0 & 144 & -97\\
		0 & 0 & 0 & 0 & 0 & 0 & 0 & 12 & 0 & 0 & 0 & 0 & 0 & 0 & -18 & 11 \\
		0 & 0 & 0 & 0 & 0 & 0 & 0 & 0 & 12 & 0 & 0 & 0 & 0 & 0 & -18 & 11 \\
		0 & 0 & 0 & 0 & 0 & 0 & 0 & 0 & 0 & 16 & 0 & 0 & 0 & 0 & 0 & -3 \\
		0 & 0 & 0 & 0 & 0 & 0 & 0 & 0 & 0 & 0 & 48 & 0 & 0 & 0 & 48 & -61 \\
		0 & 0 & 0 & 0 & 0 & 0 & 0 & 0 & 0 & 0 & 0 & 24 & 0 & 0 & -12 & 7 \\
		0 & 0 & 0 & 0 & 0 & 0 & 0 & 0 & 0 & 0 & 0 & 0 & -4 & 0 & -4 & 3 \\
		0 & 0 & 0 & 0 & 0 & 0 & 0 & 0 & 0 & 0 & 0 & 0 & 0 & 24 & 12 & -13 
	\end{pmatrix}
	\begin{pmatrix}
		\zeta(6) \\
		\zeta(1,5) \\
		\zeta(2,4) \\
		\zeta(3,3) \\
		\zeta(4,2) \\
		\zeta(1,1,4) \\
		\zeta(1,2,3) \\
		\zeta(1,3,2) \\
		\zeta(2,1,3) \\
		\zeta(2,2,2) \\
		\zeta(3,1,2) \\
		\zeta(1,1,1,3) \\
		\zeta(1,1,2,2) \\
		\zeta(1,2,1,2) \\
		\zeta(2,1,1,2) \\
		\zeta(1,1,1,1,2) 
	\end{pmatrix}
	=
	\begin{matrix}
		\boldsymbol{0}.
	\end{matrix}
	$
	
	\section*{Acknowledgemet}
	The author would like to thank my supervisor Professor Hirofumi Tsumura for his kind advice and helpful comments. Also, the author would like to thank Professor Yasushi Komori for valuable comments. This work was supported by JST, the establishment of university fellowships towards the creation of science technology innovation, Grant Number JPMJFS2139.

	\kg
	
	{\bf Address:} Department of Mathematical Sciences, Tokyo Metropolitan University, 1-1 Minami-osawa, Hachioji City, Tokyo, 192-0364 Japan.
	
	{\bf E-mail:} nishibiro-kyousuke@ed.tmu.ac.jp


\begin{bibdiv}
		\begin{biblist}
			\bib{AK1}{article}{
				author={T. Arakawa},author={M. Kaneko},
				title={Multiple zeta values, poly-Bernoulli numbers, and related zeta functions},
				journal={Nagoya Math. J.},
				volume={153},
				date={1999},
				number={},
				pages={189--209},
				issn={},
			}
			\bib{AO1}{article}{
				AUTHOR = {T. Aoki}, author={Y. Ohno},
				TITLE = {Sum relations for multiple zeta values and connection formulas
					for the {G}auss hypergeometric functions},
				JOURNAL = {Publ. Res. Inst. Math. Sci.},
				VOLUME = {41},
				YEAR = {2005},
				NUMBER = {2},
				PAGES = {329--337},
			}
			\bib{BS}{article}{
				author={Y. Baba},author={M. Nakasuji},author={M. Sakata}
				title={Multi-indexed poly-Bernoulli numbers},
				journal={arXiv:2211.14549},
				volume={},
				date={},
				number={},
				pages={},
				issn={},
			}
			\bib{Bro}{article}{
				AUTHOR = {Brown, Francis C. S.},
				TITLE = {Multiple zeta values and periods of moduli spaces
					{$\overline{\germ M}_{0,n}$}},
				JOURNAL = {Ann. Sci. \'{E}c. Norm. Sup\'{e}r. (4)},
				VOLUME = {42},
				YEAR = {2009},
				NUMBER = {3},
				PAGES = {371--489},
			}
			\bib{Go}{article}{
				AUTHOR = {A. B. Goncharov},
				TITLE = {Multiple polylogarithms and mixed tate motives},
				JOURNAL = {preprint},
				VOLUME = {},
				YEAR = {2001},
				NUMBER = {},
				PAGES = {arXiv:math/0103059v4},
			}
			\bib{HIST}{article}{
				AUTHOR = {M. Hirose},author={K. Iwaki},author={N. Sato},author={K. Tasaka},
				TITLE = {Duality/sum formulas for iterated integrals and their
					application to multiple zeta values},
				JOURNAL = {J. Number Theory},
				VOLUME = {195},
				YEAR = {2019},
				PAGES = {72--83},
			}
			\bib{It1}{thesis}{
				author={K. Ito},
				title={On a multi-variable Arakawa-Kaneko zeta function for non-positive or positive indices},
				school={Tohoku University, doctoral thesis},
				year={2020},
			}
			\bib{It2}{article}{
				AUTHOR = {K. Ito},
				TITLE = {The multi-variable {A}rakawa-{K}aneko zeta function for
					non-positive indices and its values at non-positive integers},
				JOURNAL = {Comment. Math. Univ. St. Pauli},
				FJOURNAL = {Commentarii Mathematici Universitatis Sancti Pauli},
				VOLUME = {68},
				YEAR = {2020},
				PAGES = {69--82},
			}
			\bib{It3}{article}{
				AUTHOR = {K. Ito},
				TITLE = {Analytic continuation of multi-variable {A}rakawa-{K}aneko
					zeta function for positive indices and its values at positive
					integers},
				JOURNAL = {Funct. Approx. Comment. Math.},
				VOLUME = {65},
				YEAR = {2021},
				NUMBER = {2},
				PAGES = {237--254},
			}
			\bib{KT1}{article}{
				AUTHOR = {M. Kaneko},author={H. Tsumura},
				TITLE = {Multi-poly-{B}ernoulli numbers and related zeta functions},
				JOURNAL = {Nagoya Math. J.},
				FJOURNAL = {Nagoya Mathematical Journal},
				VOLUME = {232},
				YEAR = {2018},
				PAGES = {19--54},
			}
			\bib{KO}{article}{
				AUTHOR = {N. Kawasaki},author={Y. Ohno},
				TITLE = {Combinatorial proofs of identities for special values of
					{A}rakawa-{K}aneko multiple zeta functions},
				JOURNAL = {Kyushu J. Math.},
				VOLUME = {72},
				YEAR = {2018},
				NUMBER = {1},
				PAGES = {215--222},
			}
			\bib{Kom1}{article}{
				AUTHOR = {Y. Komori},
				TITLE = {An integral representation of multiple {H}urwitz-{L}erch zeta
					functions and generalized multiple {B}ernoulli numbers},
				JOURNAL = {Q. J. Math.},
				VOLUME = {61},
				YEAR = {2010},
				NUMBER = {4},
				PAGES = {437--496},
			}
			\bib{KoT}{article}{
				AUTHOR = {Y. Komori},author={H. Tsumura},
				TITLE = {On {A}rakawa-{K}aneko zeta-functions associated with
					{$GL_2(\Bbb C)$} and their functional relations},
				JOURNAL = {J. Math. Soc. Japan},
				VOLUME = {70},
				YEAR = {2018},
				NUMBER = {1},
				PAGES = {179--213},
			}
			\bib{Lap}{article}{
				AUTHOR = {J. A. Lappo-Danilevsky},
				TITLE = {Th\'eorie algorithmique des corps de Riemann},
				JOURNAL = {Rec. Math. Moscou},
				VOLUME = {34},
				YEAR = {1927},
				NUMBER = {6},
				PAGES = {113--146},
			}
			\bib{Pan}{thesis}{
				AUTHOR = {E. Panzer},
				TITLE = {Feynman integrals and hyperlogarithms},
				school={Humboldt-Universität zu Berlin}
				YEAR = {2014},
			}
			\bib{Ya}{article}{
				AUTHOR = {S. Yamamoto},
				TITLE = {Multiple zeta functions of {K}aneko-{T}sumura type and their
					values at positive integers},
				JOURNAL = {Kyushu J. Math.},
				VOLUME = {76},
				YEAR = {2022},
				NUMBER = {2},
				PAGES = {497--509},
			}
		\end{biblist}
	\end{bibdiv}
\end{document}